\documentclass[12pt]{article}
\usepackage{amsmath,amsthm,amsfonts,amssymb}
\usepackage[T1,T2A]{fontenc}
\usepackage[utf8]{inputenc}
\usepackage[english]{babel}
\usepackage[utf8]{inputenc}
\usepackage{tikz-cd}
\usepackage{mathrsfs}
\usepackage{authblk}

\newcommand{\mC}{\mathbb{C}}

\newcommand{\mR}{\mathbb{R}}
\newcommand{\mT}{\mathbb{T}}
\newcommand{\mZ}{\mathbb{Z}}
\newcommand{\mN}{\mathbb{N}}
\newcommand{\wh} {\widehat}
\newcommand{\bra}{\langle}
\newcommand{\cet}{\rangle}

\newcommand{\calB}{{\cal B}}

\newcommand{\calD}{{\cal D}}

\newcommand{\calH}{{\cal H}}
\newcommand{\calI}{{\cal I}}

\newcommand{\calL}{{\cal L}}

\newcommand{\calP}{{\cal P}}

\newcommand{\calS}{{\cal S}}

\newcommand{\calU}{{\cal U}}

\newcommand{\calX}{{\cal X}}

\newcommand{\calCU}{{\mathcal{CU}}}
\newcommand{\calSB}{{\mathcal{SB}}}

\newcommand{\one}{{\bf 1}}
\newcommand{\bA}{{\bf A}}
\newcommand{\bB}{{\bf B}}

\newcommand{\bG}{{\bf G}}

\newcommand{\bU}{{\bf U}}

\newcommand{\mfh}{{\mathfrak{h}}}
\newcommand{\Ker}{{\text{Ker}}}
\newcommand{\mfH}{{\mathfrak{H}}}

\theoremstyle{plain}
\newtheorem{theorem}{Theorem}[section]

\newtheorem{lemma}{Lemma}[section]
\newtheorem{corollary}{Corollary}[section]

\newtheorem{definition}{Definition}[section]
\newtheorem{remark}{Remark}[section]

\makeatletter \@addtoreset{equation}{section} \makeatother

\title{Entropy of a unitary operator on $l^2(\mN_0)$}
\author[1]{{Andrey Chernyshev}}
\affil[1]{Steklov
	Mathematical Institute of Russian Academy of Sciences}

\begin{document}
	
\maketitle
	
\begin{abstract}
In this article we continue to study the concept of entropy introduced in \cite{AfonTres22}, \cite{TresCher22}-\cite{Tres20}. We calculate entropy for a wider class of finite-dimensional operators in comparison with \cite{TresCher22}. We also approximate the entropy of a unitary operator on $l^2(\mN_0)$ by the entropy of finite-dimensional operators. Finally we calculate the entropy of several operators on $l^2(\mN_0)$.
\end{abstract}
	\section{Introduction and main definitions}
	 In classical dynamics, chaos is associated with such concepts as non-integrability, positivity of the Lyapunov exponents, symbolic dynamics, positivity of topologic and metric entropy. There are many attempts to propose analogous definitions for quantum chaos \cite{Accardi,Alicki,
	 	Down1,Connes,Down2,makarov} but here the situation is not so clear. In this article we measure quantum chaos by using a certain quantities similar to the metric entropy. Since quantum evolution is unitary, we mainly discuss the entropy of the unitary operator.
	
	Let $\calX$ be a non-empty set, and let $\calB$ be a $\sigma$-algebra of subsets $X \subset \calX$. Consider the measure space $(\calX,\calB,\mu)$, where $\mu$ is a probability measure, $\mu(\calX)=1$. We say that $\chi=\{Y_1,\ldots,Y_J\}$ is a partition of $\calX$ if \begin{gather*}
		Y_j \in \calB,\qquad \mu\Big(\calX \setminus \bigcup_{j=1}^J Y_j\Big) =0, \qquad\mu(Y_j\cap Y_k) =0,\qquad j\neq k.
	\end{gather*}
	We say that $\kappa=\{X_0,\ldots,X_K\}$ is a subpartition of $\chi=\{Y_0,\ldots,Y_J\}$ if for any $k \in \{0,\ldots,K\}$ there exist $j \in \{0,\ldots,J\}$ such that $\mu(X_k \setminus Y_j)=0.$ We also say that $\kappa$ is a refinement of the partition $\chi$.
	
	Let $F\colon\mathcal{X} \to \mathcal{X}$  be an endomorphism of the space $(\mathcal{X},\mathcal{B},\mu)$. This means that for any $X \in \mathcal{B}$ the set $F^{-1}(X)$ (the full preimage of $X$) lies in $\mathcal{B}$ and $\mu(X) = \mu(F^{-1}(X)).$ Invertible endomorphisms are called automorphisms. Let $\textrm{End}(\mathcal{X})$ denote the semigroup of all endomorphisms of the space $(\mathcal{X},\mathcal{B},\mu)$. There are two standart constructions associated with an arbitrary $F \in \textrm{End}(\mathcal{X},\mu)$.
	
	(1). Each $F$ generates an isometry (a unitary operator if $F$ is an automorphism) $U_F$ on the space $L^2(\mathcal{X},\mu)$ (the Koopman operator):
	\begin{equation*}\label{koop}
		L^2(\mathcal{X},\mu) \ni f \mapsto U_F f = f \circ F,\qquad U_F= \textrm {Koop} (F).
	\end{equation*}

	(2). For any $F \in End(\mathcal{X},\mu)$ one can calculate its metric entropy (also known as the Kolmogorov-Sinai entropy) $h(F)$.
	
	Let us recall how metric entropy can be constructed for an arbitrary endomorphism. Let $ \mathcal{S}_{n,K}$ be the set of all mappings $\{0,\ldots,n\} \to \{0,\ldots,K\}$. For any partition $\chi = \{ X_0, \ldots, X_K\}$ and for any $\sigma\in \mathcal{S}_{n,K}$ we define a measurable set $\mathbf{X}_{\sigma} \subset \mathcal{X}$ by 
	\begin{displaymath}
		\mathbf{X}_{\sigma} = F^{-n}(X_{\sigma(n)}) \cap \ldots \cap F^{-1}(X_{\sigma(1)})\cap X_{\sigma(0)}.
	\end{displaymath}
	
	Define  
	\begin{displaymath}
		h_F (\chi,n+1) = - \sum_{\sigma \in \mathcal{S}_{n,K}} \mu(\mathbf{X}_{\sigma}) \log \mu(\mathbf{X}_{\sigma}).
	\end{displaymath}
	
	It is well known (see \cite{katokHess}) that $h_F$ is subadditive as a function of the second argument: $h_F(\chi,n+m) \leq h_F(\chi,n) + h_F(\chi,m).$ Therefore there exists the limit
	\begin{displaymath}
		h_F(\chi) = \lim_{n \to \infty} \frac{1}{n} h_F (\chi,n).
	\end{displaymath}
	
	It was proved in \cite{katokHess} that the function $h_F(\chi)$ does not decrease as the partition  $\chi$ is refined. The metric entropy of the endomorphism $F$ is defined by
	\begin{displaymath}
		h(F) = \sup_{\chi} h_F (\chi).
	\end{displaymath}
	
		In \cite{AfonTres22,TresCher22,Tres21,Tres20} the question was raised about defining a function $\mathfrak{h}$ on the semigroup  $\textrm{Iso}(L^2(\mathcal{X},\mu))$ of isometric operators such that the diagram
	$$
	\begin{array}{rcl}
		& \textrm{End}(\mathcal{X},\mu)& \\
		h \swarrow& &\searrow \textrm{Koop} \\
		\overline{\mathbb{R}}_{+}  & \stackrel{\mathfrak{h}}{\longleftarrow}  & \textrm{Iso}(L^2(\mathcal{X},\mu))
	\end{array}$$
	is commutative.
	
	An entropy of a unitary operator $\mfh(U)$ was constructed in \cite{AfonTres22,Tres21,Tres20}  by using the concept of the $\mu$-norm.
	
	Consider the Hilbert space $\mathcal{H}=L^{2}(\mathcal{X},\mu)$ with the scalar product and norm
	\begin{equation*} 
		\langle f,g \rangle=\int_{\mathcal{X}} f \overline{g} \, d\mu,\qquad   \|f\|=\sqrt{\langle f,g \rangle}.
	\end{equation*}
	
	Let $\calL(\calH)$ denote the semigroup of bounded linear operators on $\calH$. For any $W\in \calL (\calH)$  let $\|W\|$ be its $L^2$ operator norm:
	\begin{displaymath}
		\|W\| = \sup_{\|f\|=1} \|Wf\|.
	\end{displaymath}
	
	For any $X\in \mathcal{B}$, consider the orthogonal projection
	\begin{equation*}
		\label{eq1}
		\widehat{\mathbf{1}}_X\colon\mathcal{H}\to\mathcal{H}, \qquad \mathcal{H}\ni f \mapsto \widehat{\mathbf{1}}_X f =\mathbf{1}_X \cdot f ,    
	\end{equation*}
	where $\mathbf{1}_X$ is the indicator of the set $X$.

	Let $W$ be a bounded operator on $\mathcal{H}$. For any partition $\chi=\{Y_0,...,Y_J\}$ of the set $\mathcal{X}$, we define 
	\begin{equation} \label{eq11}
		\mathcal{M}_{\chi} (W) = \sum_{j=0}^J\mu(Y_j)\|W\widehat{\mathbf{1}}_{Y_j}\|^2.
	\end{equation}
	
	We define the $\mu$-norm of the operator $W$ by the equation: 
	\begin{equation} \label{munorm}
		\|W\|_\mu = \inf_{\chi} \sqrt{\mathcal{M}_\chi(W)}.
	\end{equation}

	The group of unitary operators on $\mathcal{H}$ is denoted by $\mathcal{U}(\mathcal{H})$. Let $\chi = \{X_0,\ldots,X_K\}$ be a partition of $\mathcal{X}$, and let $\sigma \in \mathcal{S}_{n,K}$.  In \cite{Tres21} the following definition of entropy $\mathfrak{h}_{1}(U)$ of an operator $U \in \mathcal{U}(\mathcal{H})$ was proposed. Consider
	
	\begin{equation}\label{Xsigma}
		\mathfrak{X}_{\sigma} = \widehat{\mathbf{1}}_{X_{\sigma(n)}} U \widehat{\mathbf{1}}_{X_{\sigma(n-1)}} U \ldots U \widehat{\mathbf{1}}_{X_{\sigma(0)}}.
	\end{equation}
	By definition, put
	\begin{equation*} \label{eq15}
		\mathfrak{h}_{1}(U,\chi,n) = - \sum_{\sigma \in \mathcal{S}_{n,K}} \|\mathfrak{X}_{\sigma} \| ^2 _{\mu} \log  \|\mathfrak{X}_{\sigma} \| ^2 _{\mu},
	\end{equation*}
	\begin{equation} \label{eq16}
		\mathfrak{h}_{1}(U,\chi) = \lim_{n \to \infty} \frac{1}{n} \mathfrak{h}_{1}(U,\chi,n), \qquad \mathfrak{h}_{1}(U) = \sup_{\chi} \mathfrak{h}_{1}(U,\chi) .
	\end{equation}
	
	If the limit (\ref{eq16}) exists, then $\mathfrak{h}_{1}(U)$ is called the entropy of the operator $U \in \mathcal{U}(\mathcal{H})$. It is known \cite{Tres21} that $\mathfrak{h}_{1}(U_F) = h(F)$ for $U_F=\textrm{Koop}(F)$ where F is arbitary authomorphism of $(\calX,\calB,\mu)$. Unfortunately it is unknown whether the limit (\ref{eq16}) always exist. Is the function $\mathfrak{h}_{1}(U,\chi)$ monotone as the partition is refined? Another definition of the entropy $\mathfrak{h}(U)$ of an operator $U \in \mathcal{U}(\mathcal{H})$ was proposed in \cite{AfonTres22,TresCher22} (see also Section \ref{definitionoftheentropy}) such that the equation $\mathfrak{h}(U_F) = h(F)$ holds for any Koopman operator $U_F=\textrm{Koop}(F)$. The construction of the entropy $\mathfrak{h}(U)$ is based on the fact that in several important situations there is a natural way to associate with any $U\in \calU_{\mu}(\calH)$ a bistochastic operator $\bB=b(U)$. Then it is possible to propose a definition of the entropy $\mfh(U)$ ideologically close to constructions in \cite{Down1, Down2}. The disadvantage of the second definition of the entropy $\mfh(U)$ is that the mapping $U \mapsto b(U)$ has been defined only for some special classes of unitary operators.
	
	 In general, the entropy means how much information is lost when applying a unitary operator. There is another interpretation in finite-dimensional case $\calX < \infty$. In this case it is possible to describe a subgroup of operators with zero entropy explicitly. Thus this entropy in a certain sense means the distance to this subgroup.  
	
	In the literature there exist several attempts to extend the concept of the measure entropy to quantum systems, see \cite{Accardi,Alicki,Connes,makarov} and many others. In \cite{Accardinotes} several mutual relations between
	these approaches are given. In \cite{Down1,Down2} a construction for the measure entropy is proposed
	for doubly stochastic (bistochastic) operators on various spaces of functions on a measure
	space.

	In \cite{TresCher22} entropy $\mfh(U)$ of the finite-dimensional unitary operator $U \in \calU (\mC^J)$ is calculated by \begin{equation}\label{previousresult}
		\mfh(U) = - \sum_{j,k =0}^{J-1} \alpha_k \bB_{jk} \log \bB_{jk}, \qquad \bB = b(U),
	\end{equation}
	where 
	\begin{equation*}
		\alpha_k = \big(\calP_{\text{Ker}(\bB - I)} \mu \big)_k,\qquad \mu = (\mu_0,\ldots,\mu_{J-1}).
	\end{equation*}
	
	The equation (\ref{previousresult}) motivates the problem of approximation of the entropy $\mfh(U)$, $U \in \calU_{\mu}(\calH)$ by the entropy $\mfh(U_J)$ defined by (\ref{previousresult}) of a finite-dimensional approximation $U_J \in \calL(\mC^J)$.
	
	In this paper, we propose some process of finite-dimensional approximation of the unitary operator $U\in \calU_{\mu}(\calH)$. In the other words, we construct some finite-dimensional sequence $\{U_J\}_{J\in \mN_0}$ which converges (in some sense) to $U \in \calU_{\mu}(\calH)$. We also establish some connection between the sequence $\{\mfh(U_J)\}_{J\in \mN_0}$ and $\mfh(U)$.
	
	Consider $\calX = \mN_0=\{0,1,\ldots\}$. Let $\mu_j > 0$ be the measure of $j$-th point, $ \sum_{j \in \mN_0} \mu_j = 1.$  Then $\calH = l^2(\mN_0,\mu)$ with the scalar product and norm \begin{equation}\label{scalarproductandnorm}
		\bra x,y \cet = \sum_{j\in \mN_0} \mu_j x_j \overline{y}_j,\qquad \| x \| =\sqrt{ \bra x,x \cet}.
	\end{equation}

We also consider the space $l^1(\mN_0)$ with the norm \begin{equation*}
	\| x\|_{l^1(\mN_0)} = \sum_{j\in\mN_0} |x_j|,\qquad x\in l^1(\mN_0).
\end{equation*}

	Let $(W_{jk})_{j,k\in \mN_0}$ be the representation of the operator $W$ in the basis $\{e_j = (0,\ldots,\mu_j ^{-1/2},\ldots)\}_{j\in\mN_0}$. Below, we establish the following facts.
	
	(1). The $\mu$-norm of an operator $W \in \calL(\calH)$ is calculated in Section \ref{sectionmunorm} by 
	\begin{equation*}
		\|W \|^2_{\mu} = \sum_{j,k \in \mathbb{N}_0} \mu_k |W_{kj}|^2.
	\end{equation*}

	In Section \ref{sectionmunorm} we also prove that for arbitrary partitions $\chi=\{X_n\}_{n=0}^p$ and $\kappa =\{Y_k\}_{k=0}^q$ of the set $\calX$ 
	\begin{equation*}
		\| W\|_{\mu} ^2 =\sum_{n=0}^p \sum_{k=0}^q \| \widehat{\one}_{X_n} W \widehat{\one}_{Y_k} \|^{2} _{\mu}.
	\end{equation*}

	(2). In Section \ref{semigroupbist} we introduce semigroups of contraction $\calCU_{\mu}(\calH)$ and semibistochastic $\calSB(\mN_0)$ operators by \begin{equation*}
		\calCU_{\mu}(\calH) = \{U \in \calL(\calH)\colon \forall \; x\in \calH,\quad \| U x \| \leq \|x\|\},
	\end{equation*}
\begin{equation*}
	\calSB(\mN_0) = \bigg\{\bB=(\bB_{jk})_{j,k\in \mN_0} \in \calL(l^1(\mN_0))\colon \bB_{jk} \geq 0,\quad \sum_{j\in \mN_0} \bB_{jk} \leq 1,\quad \sum_{k\in \mN_0} \bB_{jk} \leq 1\bigg\}.
\end{equation*}
  We establish that $\calSB(\mN_0)$ is a closed semigroup. We use $\mu$-norm to construct the mapping $b\colon \calCU_{\mu}(\calH) \to \calSB(\mN_0)$. Suppose $U\in \calCU_{\mu}(\calH)$. Using the mapping $b$, we get $\bB=b(U) \in \calSB(\mN_0)$. We consider the following sequences \begin{equation*}
  	\alpha \in \mN_0, \qquad a_{\alpha} ^k (U) = \sum_{j_{k-1},\ldots,j_0 \in \mN_0} \bB_{j_{k-1} j_{k-2}} \bB_{j_{k-2} j_{k-3}} \ldots \bB_{j_0 \alpha},
  \end{equation*}
\begin{equation*}
		\alpha \in \mN_0 ,\quad u_{\alpha}^n (U) = \sum_{k=1}^n a_{\alpha}^k (U),\qquad u_{\alpha}(U) = \lim_{n \to \infty} \frac{1}{n} u_{\alpha}^n (U).
\end{equation*}
 We prove that for any $\alpha \in \mN_0$ and $n \in \mN_0$ the sequence of operators \begin{equation*}
 	\bB_{n,\alpha} = \frac{1}{n} \sum_{j=0}^{n-1} a_{\alpha}^{n-1 -j}(U) \bB^j.
 \end{equation*}
strongly converges to $u_{\alpha}(U)\calP_{\Ker(\bB -I)}$, where $\calP_{\Ker(\bB -I)}$ is the projection on $\Ker(\bB -I)$, i.e \begin{equation*}
		\forall x \in l^1(\mN_0)\colon \lim_{n\to \infty}\|\bB_{n,\alpha} x  - u_{\alpha}(U)\calP_{\Ker(\bB -I)} x \|_{l^1(\mN_0)} = 0 .
\end{equation*}
	
	(3).  Section \ref{definitionoftheentropy} introduces the definition of entropy of a unitary operator $\mfh(U)$. We establish that for any $F \in \text{Aut}(\mN_0,\mu)\colon$ $\mfh(U)=h(F)=\mfh_1(U)$.
	
	(4). In Section \ref{monotonicityofmfh} we prove that the function $\mfh(U,\chi)$ is monotone as the partition is refined. More precisely, let $\kappa$ be a subpartition of the partition $\chi$. Then $\mathfrak{h}(U,\chi,n) \leq \mathfrak{h}(U,\kappa,n)$.

	(5). In Section \ref{Calcofentropy} we study the finite-dimensional case. Consider the following scalar product and norm on $\mC^J\colon$ \begin{equation*}
		\langle x, y \rangle_J = \sum_{j=0}^{J-1} \mu_j x_j \overline{y_j},\qquad \|x \|_J = \sqrt{\langle x,x\rangle_J}.
	\end{equation*}
	
	Suppose $U \in \calCU_{\mu}(\mC^J)$, $\mu = (\mu_0,\mu_1,\ldots,\mu_{J-1}) \in \mC^J$, $e=(1,\ldots,1) \in \mC^J$. We prove that its entropy is calculated by \begin{equation}\label{finitedimentrofU}
	\mfh(U) = - \sum_{j,k=0 }^{J-1} \big(\calP^T_{\text{Ker}(\bB-I)} e\big)_j \bB_{jk} \big(\calP_{\text{Ker}(\bB-I)} \mu\big)_k  \log \bB_{jk},\quad b(U)=\bB,
	\end{equation}
	
	(6). In Section \ref{propmfh} we prove the following properties of $\mfh(U)$:
	
	(i)   Suppose $U \in \calCU_{\mu}(\mC^J)$ and $\calD_1, \calD_2 \in \calU(\mC^J)$ are diagonal operators. Then  $\mfh(\calD_1 U \calD_2) = \mfh(U)$.
	
	(ii) Suppose $U \in \calCU_{\mu}(\mC^J)$, $F \in \text{Aut}(\mN_0)$. Then $\mfh(U_F ^{-1} U U_F) = \mfh(U)$.  Informally speaking, this means that measure preserving coordinate changes on $\calX$ preserve $\mfh(U)$.
	
	(iii) We establish the specific operators with zero entropy.
	
	(7). In Section \ref{approximationofentropy} we introduce a process of approximation $\mfh(U)$. Suppose $U\in \calU_{\mu}(\calH)$. Consider the following operators $p_J\colon l^2(\mN_0,\mu) \to \mC^J$, $q_J\colon \mC^J \to l^2(\mN_0,\mu)$ defined by \begin{gather*}
		p_J\colon (x_0,\ldots,x_{J-1},x_J,\ldots) \mapsto (x_0,\ldots,x_{J-1}),\\
		 q_J\colon (x_0,\ldots,x_{J-1}) \mapsto (x_0,\ldots,x_{J-1},0,\ldots).
	\end{gather*} The sequence $\{U_J\in \calCU_{\mu}(\mC^J)\}_{J \in \mN_0}$ is constructed by the unitary operator $U$ so that the diagram \begin{equation*}
		\begin{tikzcd}
			\mC^J\arrow{d}{q_J} \arrow{r}{U_J}& \mC^J \\
			l^2(\mN_0,\mu)\arrow{r}{U} & l^2(\mN_0,\mu) \arrow{u}{p_J}
		\end{tikzcd}
	\end{equation*}
	is commutative. We prove that the sequence $\{U_J\}_{J \in \mC^J}$ strongly converges to $U$.  This sequence generates the sequence $\{\mfh(U)\}_{J\in \mN_0}$, where the numbers $\mfh(U_J)$ are defined by (\ref{finitedimentrofU}). We also prove for any $J\in \mN_0$ the inequality $\mfh(U_J) \leq \mfh(U_{J+1})$.
	
	(8). In Section \ref{entrexamples} we define the number $\mfH(U)$ by \begin{equation*}
		\mfH(U) = \lim_{J\to \infty} \mfh(U_J).
	\end{equation*}
	
	Then we establish some connection between the numbers $\mfh(U)$ and $\mfH(U)$. In particular, we obtain the following statements
	
	(i) Let $D\in \calCU_{\mu}(\calH)$ be a diagonal operator. Then $\mfH(D) = \mfh(D)=0.$
	
	(ii) Let $F\colon \mN_0 \to \mN_0$ be an automorphism. Then $\mfH(U_F) = \mfh(U_F)=0.$
	
	(iii)	We also consider a strict contraction operator $U \in \calCU_{\mu}(\calH)\colon$ for all $x \in \calH$ $\| Ux \| \leq q \| x \|,$ $0<q<1$. We prove that for this operator: $\mfH(U)=0$. 
	
	(vi) The equation $\mfH(U)=\mfh(U)$ does not hold in general case. We present an example of operator $A\in \calCU_{\mu}(\calH)$ such that $\mfh(A) = \log 2$, but $\mfH(A) = 0$. 
	
	(vii) We also construct examples of operators $\bB_n$ and $\bB_{\infty}$ when $\mfH(\bB_n) = \log n,$  $\mfH(\bB_\infty) = \infty.$
	
	 $\textbf{Acknowledgements}$.	The author is greateful to professor D.V. Treschev for constant attention to this work and to K.A. Afonin for useful discussions.
	 
   Chapters 2-6 of this work were supported by the
the Russian Science Foundation under grant no. 19-71-30012,
https://rscf.ru/en/project/19-71-30012/. 

Chapters 7-9 were carried out with the support of the Theoretical Physics
and Mathematics Advancement Foundation "BASIS".

	\section{$\mu$-Norm of an operator} \label{sectionmunorm}
	\subsection{$\mu$-Norm and partitions}
	Consider $\calX=\mN_0$. Let $\mu_j>0$ be the measure of the $j$-th point, $\sum_{j\in \mN_0} \mu_j =1$. Then $L^2(\calX,\mu) \cong l^2(\mN_0,\mu)$ with the scalar product and norm are defined by (\ref{scalarproductandnorm}). Let $W$ be a linear operator on $l^2(\mN_0,\mu)$. Let $(W_{jk})_{j,k\in \mN_0}$ be the representation of the operator $W$ in the basis $\{e_j = (0,\ldots,\mu_j ^{-1/2},\ldots)\}_{j\in\mN_0}$. We define for any $X \subset \mathbb{N}_0$
	\begin{equation*}
		\widehat{\mathbf{1}}_X\colon  \l^2(\mN_0,\mu) \to \l^2(\mN_0,\mu),\qquad \widehat{\mathbf{1}}_X\colon (x_0,x_1,\ldots) \mapsto \big(\mathbf{1}_X(0) x_0,\mathbf{1}_X(1) x_1,\ldots \big),
	\end{equation*}
	where $\mathbf{1}_X$ is the indicator of the set $X$.
	
	The function $\|\cdot \|_{\mu}\colon \calL(\calH) \to \mR_{+}$ was studied in \cite{Tres20} for any probability space $(\calX,\calB,\mu)$.
	\begin{lemma} \label{subpartition}
		Let $\chi'$ be a subpartition of $\chi$ then $\mathcal{M}_{\chi'}(W) \leq \mathcal{M}_{\chi}(W)$. Hence the quantity $\mathcal{M}_{\chi}(W)$ approaches the infimum  (\ref{eq11}) on the finest partition.
	\end{lemma} 
	\begin{proof}
		This Lemma was proved in \cite{Tres20}.
	\end{proof}

	\subsection{Computation of the $\mu$-norm}
	\begin{lemma} 
		Let $W\in \mathcal{L}(\mathcal{H})$. Then for any $j \in \mN_0$ \begin{equation*}
			\|W \widehat{\mathbf{1}}_{\{j\}}\|^2 = \frac{1}{\mu_j} \sum_{k \in \mathbb{N}_0} \mu_k |W_{kj}|^2.
		\end{equation*}
	\end{lemma}
	
	\begin{proof}
		Let $x \in \mathcal{H}$. For any $k\in \mathbb{N}_0$: \begin{equation*}
			\big(W \mathbf{1}_{\{j\}} x \big)_k = \sum_{\alpha,\beta \in \mathbb{N}_0} W_{k \alpha} \big(\widehat{\mathbf{1}}_{\{j\}}\big)_{\alpha \beta}\; x_{\beta} = W_{kj} x_j.
		\end{equation*}
		
		Thus we have \begin{equation*}
				\|W \widehat{\mathbf{1}}_{\{j\}} x \|^2 = \langle W\widehat{\mathbf{1}}_{\{j\}} x, \widehat{\mathbf{1}}_{\{j\}}x  \rangle = |x_j|^2\sum_{k \in \mathbb{N}_0} \mu_k |W_{kj}|^2  \leq \| x\|^2 \frac{1}{\mu_j} \sum_{k \in \mathbb{N}_0} \mu_k |W_{kj}|^2. 
		\end{equation*}
		
		To conclude the proof, it remains to note that equality is achieved on the following vector: $\mathcal{H} \ni x_{*} = (0,\ldots,0,\mu_j ^{-1/2},0,\ldots),$ $\|x_{*} \| =1.$
	\end{proof}
	
	\begin{lemma} \label{lemmunorm}
		Let $W \in \mathcal{L}(\mathcal{H})$. Then \begin{equation}\label{munorm}
			\|W \|^2_{\mu} = \sum_{j,k \in \mathbb{N}_0} \mu_k |W_{kj}|^2, \qquad \|W\|_{\mu} \leq \| W\| < \infty.
		\end{equation}
	\end{lemma}
	\begin{proof}
		Let $\chi=\{Y_{0},\ldots,Y_n\}$ be an arbitrary partition of $\calX$. Then \begin{equation*}
			\|W \|^2_{\mu} \leq \mathcal{M}_{\chi}(W) = \sum_{j=0}^{n} \mu(Y_j) \| W \widehat{\mathbf{1}}_{Y_j}\|^2 \leq \sum_{j=0}^{n} \mu(Y_j) \|W \|^2 = \|W\|^2 <\infty.
		\end{equation*}
		
		Let $\chi_n = \{\{0\},\{1\},\ldots, \{n\},X_{n+1}\}$ be the partition of $\mathbb{N}_0$, where $X_{n+1} = \{n+1,n+2,\ldots\}$. Using Lemma \ref{subpartition}, we get \begin{equation*}
			\|W \|^2 _{\mu} = \inf_{\chi} \mathcal{M}_{\chi}(W) = \lim_{n \to \infty} \mathcal{M}_{\chi_n}(W).
		\end{equation*}
		
		Thus we have 
		\begin{equation*}
			\mathcal{M}_{\chi_n} (W) = \sum_{j=0}^n \mu_j \|W \widehat{\mathbf{1}}_{\{j\}} \|^2 + M_{n+1} \|W \widehat{\mathbf{1}}_{X_{n+1}} \|^2,
		\end{equation*}
		where $M_{n+1} = \sum_{j=n+1}^{\infty} \mu_j$. Since for all $n \in \mathbb{N}_0$: $\|W \widehat{\mathbf{1}}_{X_{n+1}}\| \leq \| W\| < \infty$ and $M_{n+1} \to 0$  $(n \to \infty)$, it follows that (\ref{munorm}). 	\end{proof}

	\subsection{Koopman operator}
	
	Let the map $F\colon \mN_0 \to \mN_0$ be an authomorphism ($F \in \text{Aut}(\mN_0,\mu)$).
	Consider the Koopman operator $U_F\colon  x \mapsto x \circ F$ on $l^2(\mN_0,\mu)$, where $x \in l^2(\mN_0,\mu)$. Then \begin{equation} \label{koopmanoperator}
		(U_F)_{jk} = \sqrt{\frac{\mu_k}{\mu_j}} \delta_{F(j)k}.
	\end{equation}
	
	In the other words, the Koopman operator $U_F$ permutes points with the same measures. Consider sets $A_j,\; j\in \mN_0$ such that \begin{equation*}
		\alpha_j \in (0,1),\quad \sum_{j\in\mN_0} \alpha_j \leq 1,\qquad A_j = \{k \in \mN_0 \colon \mu_k = \alpha_j\}.
	\end{equation*}
	
	It is clear that $A_j \cap A_k = \emptyset$ for $j\neq k$ and $\#A_j < \infty$. Let $F$ be an automorphism. Then $F^{-1}(A_j)=A_j$ and for all $n\in \mN_0\colon$ $F^{-n}(A_j)=A_j$. Thus we have for all $j \in A_n$ \begin{equation} \label{muusefuleq}
	 \mu_j=\mu_{F^{-1}(j)}. 
	\end{equation}

	\begin{lemma} \label{lemconnecttwodef}
		Let $U_F$ be a Koopman operator. Consider the ordered collection of the diagonal operators $\bG = \{\widehat{g_0},\ldots,\widehat{g_n}\}$. Then \begin{equation*}
			\|\widehat{g_0} U_F \widehat{g_1} \ldots U_F \widehat{g_n}\|_{\mu} ^2 = \sum_{j \in \mN_0} |(g_0)_j|^2 |(g_1)_{F(j)}|^2 \ldots |(g_n)_{F^{n}(j)}|^2 \mu_{F^n (j)}.
		\end{equation*}
	\end{lemma}
	\begin{proof}
		Let us calculate the elements of the matrix $\widehat{g_0} U_F \widehat{g_1} \ldots U_F \widehat{g_n}$, \begin{equation*}
			\begin{split}
				\big(\widehat{g_0} U_F \widehat{g_1} \ldots U_F \widehat{g_n}\big)_{jk} &= \sum_{\alpha_1,\ldots \alpha_{n-1} \in \mN_0} (\widehat{g_0})_{j} (U_F)_{j \alpha_1} (\widehat{g_1})_{\alpha_1}  \ldots (U_F)_{\alpha_{n-1} k} (\widehat{g_n})_{k}\\
				&=\sum_{\alpha_1,\ldots \alpha_{n-1} \in \mN_0} (\widehat{g_0})_{j} \sqrt{\frac{\mu_{\alpha_1}}{\mu_j}} \delta_{F(j)\alpha_1} (\widehat{g_1})_{\alpha_1}  \ldots \sqrt{\frac{\mu_k}{\mu_{\alpha_{n-1}}}} \delta_{F(\alpha_{n-1})k} (\widehat{g_n})_{k}\\
				&= \sum_{\alpha_1,\ldots \alpha_{n-1} \in \mN_0} \sqrt{\frac{\mu_k}{\mu_j}} (\widehat{g_0})_{j} (\widehat{g_1})_{\alpha_1} \ldots (\widehat{g_n})_{k} \delta_{F(j)\alpha_1} \ldots \delta_{F(\alpha_{n-1})k} \\
				& = \sqrt{\frac{\mu_{F^n (j)}}{\mu_j}} (g_0)_j (g_1)_{F(j)}\ldots (g_n)_{F^n(j)} \delta_{F^n (j) k}.
			\end{split}
		\end{equation*}
		
		Using Lemma \ref{lemmunorm}, we obtain \begin{equation*}
			\begin{split}
				\|\widehat{g_0} U_F \widehat{g_1} \ldots U_F \widehat{g_n}\|_{\mu} ^2 &= \sum_{j,k \in \mN_0} \mu_j \frac{\mu_{F^n (j)}}{\mu_j} |(g_0)_j|^2 |(g_1)_{F(j)}|^2\ldots |(g_n)_{F^n(j)}|^2 \delta_{F^n (j) k}	\\
				&=\sum_{j \in \mN_0} |(g_0)_j|^2|(g_1)_{F(j)}|^2\ldots |(g_n)_{F^n(j)}|^2 \mu_{F^n (j)}.
			\end{split}
		\end{equation*}
	\end{proof}
	
	\subsection{Additivity}
	\begin{lemma} \label{additivity}
		Let $W \in \calL(\calH)$ and let $\chi=\{X_n\}_{n=0}^p$ and $\kappa =\{Y_k\}_{k=0}^q$ be an arbitrary partitions of the set $\calX$. Then \begin{equation*}
			\| W\|_{\mu} ^2 =\sum_{n=0}^p \sum_{k=0}^q \| \widehat{\one}_{X_n} W \widehat{\one}_{Y_k} \|^{2} _{\mu}.
		\end{equation*}
	\end{lemma}
	
	\begin{proof}
		Let us calculate the entries of the matrix $\widehat{\one}_{X_n} W \widehat{\one}_{Y_k}$ 
		\begin{equation*}
			\big(\widehat{\one}_{X_n} W \widehat{\one}_{Y_k} \big)_{jm} = \one_{X_n}(j) W_{jm}\one_{Y_k}(m),\quad j,m \in \mN_0.
		\end{equation*}
		
		Using Lemma \ref{lemmunorm}, we obtain \begin{equation*}
			\begin{split}
				\sum_{n=0}^p \sum_{k=0}^q \|\widehat{\one}_{X_n} W \widehat{\one}_{Y_k}  \|_{\mu} ^2 &= \sum_{n=0}^p \sum_{k=0}^q \sum_{j \in \mN_0} \mu_j \one_{X_n}(j) |W_{jm}|^2\one_{Y_k}(m) \\
				&=\sum_{j \in \mN_0} \mu_j \sum_{n=0}^p\one_{X_n}(j) |W_{jm}|^2\sum_{k=0}^q\one_{Y_k}(m) \\
				&= \sum_{j\in\mN_0} \mu_j |W_{jm}|^2 = \| W\|^2 _{\mu}.
			\end{split}
		\end{equation*}
	\end{proof}

	\begin{corollary}
		Lemma \ref{additivity} implies the left and right additivity,  \begin{equation*}
			\|W \|^2 _{\mu} = \sum_{n\in\mN_0} \|\widehat{\one}_{X_n} W \|^2 _{\mu},\qquad \|W \|^2 _{\mu} = \sum_{k\in\mN_0} \| W \widehat{\one}_{X_k}\|^2 _{\mu}.
		\end{equation*}
	\end{corollary}
	\begin{proof}
		We obtain the first or the second equality in the corollary if we set $\kappa = \{\mN_0\}$ or $\chi= \{\mN_0\}$, respectively, in Lemma \ref{additivity}.
	\end{proof}
	
	Note that the left additivity for arbitrary probability space $(\calX,\calB,\mu)$ was
	established in \cite{Tres20}, where the following inequality was also obtained: \begin{equation*}
		\|W \|^2 _{\mu} \leq \sum_{k\in\mN_0} \| W \widehat{\one}_{X_k}\|^2 _{\mu}.
	\end{equation*}
	
	It is not clear whether the right additivity holds for any probability space $(\calX,\mu)$.

	\subsection{Examples}
	
	\begin{lemma} \label{lemmaaboutmunormdiag}
		Suppose $\widehat{g} = \text{diag}(g_0,g_1,\ldots) \in \calL(\calH)$ is a diagonal operator. Then \begin{equation*}
			\| \widehat{g}\|^{2} _{\mu} = \|g\|_{l^2(\mN_0,\mu)}^2, \qquad g=(g_0,g_1,\ldots).
		\end{equation*}
	\end{lemma}
	
	\begin{proof}
		Using Lemma \ref{lemmunorm}, we obtain \begin{equation*}
			\|\widehat{g} \|^2 _{\mu} = \sum_{j \in \mN_0} \mu_j |g_j|^2 = \| g \|^2 _{l^2 (\mN_0,\mu)}.
		\end{equation*}
	\end{proof}
	
	\begin{corollary}
		For any $X \subset \calX \colon$ $\| \widehat{\one}_{X}\|^2 _{\mu} = \mu(X).$
	\end{corollary}
	\begin{proof}
		It is clear that $\widehat{\one}_{X} = \text{diag}(\one_X (0),\one_X (1),\ldots).$ Thus we have $\| \widehat{\one}_{X}\|^2 _{\mu} = \sum_{j \in X} \mu_j = \mu(X).$
	\end{proof}

We define the right and left shift operators $T_r$ and $T_l$ by \begin{equation*}
	T_r\colon (x_0,x_1,\ldots) \mapsto (0,x_0,\ldots),\quad
	T_l\colon (x_0,x_1,\ldots) \mapsto (x_1,x_2,\ldots).
\end{equation*}
	
	\begin{lemma}
		Consider the right and left shift operators $T_r$ and $T_l$. Then \begin{equation*}
			\|T_r \|^2 _{\mu} = 1,\qquad \|T_l \|^2 _{\mu} = 1 - \mu_0.
		\end{equation*}
	\end{lemma}
	\begin{proof}

		Thus we obtain $\| T_r \|^2 _{\mu} = \sum_{j \in \mN_0} \mu_j = 1$, and $\|T_l \|^2 _{\mu} = \sum_{j=1} ^{\infty} \mu_j = 1 - \mu_0.$
	\end{proof}

	\begin{lemma}\label{g1Wg2}
		Let $W \in \mathcal{L}(\mathcal{H})$ and $\widehat{g_1}, \wh{g_2}$ are diagonal operators on $\cal{H}$. Then \begin{equation}\label{equog_1Wg_2}
			\|\wh{g_1} W \wh{g_2} \|^2 _{\mu} = \sum_{j,k\in\mN_0} \mu_j |(\wh{g_1})_j |^2 |W_{jk}|^2 |(\wh{g_2})_k|^2.	
		\end{equation}
	\end{lemma}
	\begin{proof}
		It is clear that \begin{equation}\label{eqinlemma1}
			\big(\wh{g_1}W \wh{g_2} \big)_{jk} = (\wh{g_1})_j W_{jk} (\wh{g_1})_k.
		\end{equation}
		
		Using Lemma \ref{lemmunorm} and (\ref{eqinlemma1}), we get (\ref{equog_1Wg_2}).
	\end{proof}

	\section{Semigroup of semibistochastic operators} \label{semigroupbist}
	
	Let $W \in \mathcal{L}(\mathcal{H})$, let $W^{*}$ be a adjoint operator of $W$, so for all $x,y \in \cal{H}$ we have \begin{equation*}
		\langle Wx,y \rangle = \langle x,W^{*}y \rangle.
	\end{equation*}
	Then $W_{jk}^* = \frac{\mu_k}{\mu_j} \overline{W_{kj}}.$
	
	\subsection{Mapping $b \colon \calL(\calH)\to \calL(\calH)$} \label{mappingb}
	
	\begin{definition}
		The operator $U \in \cal{L}(\cal{H})$ is called a contraction operator if $\| Ux\| \leq \| x\|$ for all $x\in \mathcal{H}$. Let $\cal{CU}_{\mu}(\cal{H})$ be the semigroup of contraction operators.
	\end{definition}
	\begin{lemma} \label{inequalitiesCU}
		Suppose $U \in \calCU_{\mu}(\calH)$. Then \begin{equation*}
			\sum_{k \in \mN_0} \mu_k | U_{k j}|^2 \leq  \mu_j, \qquad \sum_{k \in \mN_0} \mu_k^{-1} | U_{j k }|^2  \leq   \mu_j^{-1}.
		\end{equation*}
	\end{lemma}
	\begin{proof}
	Let $U^*$ be a adjoint operator of $U$. It is well known fact that $\| U\| = \|U^* \|$. Then \begin{equation*}
			\|U^{*}x \| \leq \|U^* \| \|x \| = \| U\| \|x \| \leq \| x\|.
		\end{equation*}Therefore, $U^* \in \calCU_{\mu}(\calH)$. Consider inequalities  \begin{equation} \label{ineq1}
			\bra Ux, Ux \cet = \bra x, U^*Ux\cet \leq \bra x, x \cet, \qquad \bra U^*x, U^*x \cet = \bra x, U U^*x\cet \leq \bra x, x \cet.
		\end{equation}
		
		If we replace $x$ by $e_j = (0,\ldots,0,\mu_j^{-1/2},0,\ldots)$ in (\ref{ineq1}), we obtain
		\begin{gather*}
			\bra e_j, U^* U e_j \cet =  \sum_{\alpha \in \mN_0} \frac{\mu_\alpha}{\mu_j} | U_{\alpha j}|^2 \leq \bra e_j, e_j \cet = 1,  \\
			\bra e_j, U U^* e_j \cet =  \sum_{\alpha \in \mN_0} \frac{\mu_j}{\mu_\alpha} | U_{j\alpha }|^2  \leq \bra e_j, e_j \cet = 1.
		\end{gather*}
	\end{proof}
	
	Consider the vector spaces \begin{equation} \label{l1norm}
		l_{1}(\mN_0) = \bigg\{x=(x_0,x_1,\ldots)\colon \|x\|_{l^1(\mN_0)} = \sum_{j\in \mN_0} |x_j|< \infty\bigg\},
	\end{equation} \begin{equation}\label{l1munorm}
		l^1(\mN_0,\mu) = \bigg\{x=(x_0,x_1,\ldots)\colon\| x\|_{l^1(\mN_0,\mu)} = \sum_{j\in \mN_0} \mu_j|x_j| < \infty \bigg\}.
	\end{equation}
	
	Consider a linear operator $W \in \calL(\calH)$, and operators $\bA\colon l^1 (\mN_0,\mu) \to l^1 (\mN_0)$, $\bA^{-1}\colon l^1 (\mN_0) \to l^1 (\mN_0,\mu)$ defined by \begin{gather*}
		\bA\colon (x_0,x_1,\ldots) \mapsto (\mu_0 x_0,\mu_1 x_1, \ldots),\\
		\bA^{-1}\colon (x_0,x_1,\ldots) \mapsto (\mu_0^{-1} x_0,\mu_1^{-1} x_1, \ldots).
	\end{gather*}

	Let $b$ be the map of $\calCU_{\mu}(\calH)$ to $\calL(l^1(\mN_0))$ such that
	\begin{equation} \label{themapbdef}
		b\colon U \mapsto b(U)=\bA \bU \bA^{-1} , \qquad \bU=(\bU_{jk}) = \big(|U_{jk}|^2\big).
	\end{equation} 
	
	\begin{lemma} \label{boundB}
		Suppose $U\in\calCU_{\mu}(\calH)$ and $\bB = \bA \bU \bA^{-1}$. Then \begin{equation*}
			\sum_{j\in \mN_0} \bB_{jk} \leq 1,\qquad \sum_{k\in \mN_0} \bB_{jk} \leq 1.
		\end{equation*}
	\end{lemma}
	\begin{proof}
		Using equation (\ref{themapbdef}) and Lemma \ref{inequalitiesCU}, we obtain \begin{equation*}
			\sum_{j\in \mN_0} \bB_{jk} = \sum_{j\in \mN_0} \frac{\mu_j}{\mu_k} |U_{jk}|^2 \leq 1, \qquad \sum_{k\in \mN_0} \bB_{jk} = \sum_{k \in \mN_0} \frac{\mu_j}{\mu_j} |U_{jk}|^2 \leq 1.
		\end{equation*}
	\end{proof}
	
	\subsection{Definition and properties of $\calSB(\mN_0)$} \label{secMapb}

	\begin{definition}
		We say that $\bB \colon l^1(\mN_0) \to l^1(\mN_0)$ is a semibistochastic operator ($\bB \in \calSB(\mN_0)$) if 
		\begin{equation} \label{definitionsemibist}
			\begin{split}
				&(1). \qquad\bB_{jk} \geq 0,\quad j,k \in \mN_0,\\
				&(2).\qquad \sum_{j \in \mN_0} \bB_{jk} \leq 1,\qquad \sum_{k \in \mN_0} \bB_{jk} \leq 1.
			\end{split}
		\end{equation}
		
	\end{definition}

	\begin{lemma}
		Suppose $\bB \in \calSB(\mN_0)$. Then $\| \bB\|_{l^1(\mN_0)}\leq 1.$
	\end{lemma}
	\begin{proof}
		Let $x \in l^1(\mN_0)$. Consider \begin{equation*}
			\|\bB x \|_{l^1(\mN_0)} = \sum_{j\in\mN_0} \big| (\bB x)_j\big| \leq \sum_{j\in\mN_0} \sum_{\alpha \in \mN_0} \bB_{j \alpha} |x_{\alpha}|.
		\end{equation*}
		
		Since \begin{equation*}
			\sum_{\alpha \in \mN_0}\sum_{j\in\mN_0} \bB_{j \alpha} |x_{\alpha}| \leq \|x \|_{l^1(\mN_0)} < \infty,
		\end{equation*}
		if follows that \begin{equation*}
			\|\bB x \|_{l^1(\mN_0)} \leq \sum_{j\in\mN_0} \sum_{\alpha \in \mN_0} \bB_{j \alpha} |x_{\alpha}| = \sum_{\alpha \in \mN_0}\sum_{j\in\mN_0}\bB_{j \alpha} |x_{\alpha}| \leq \|x \|_{l^1(\mN_0)}.
		\end{equation*}
	\end{proof}
	
	\begin{lemma} \label{bUissb}
		Let $U \in \calCU_{\mu}(\calH)$. Then $b(U) \in \calSB(\mN_0)$.
	\end{lemma}
	\begin{proof}
		Note that \begin{equation*}
			\bB_{jk} = \frac{\mu_j}{\mu_k} |U_{jk}|^2.
		\end{equation*}
		
		Using Lemma \ref{inequalityUalpha}, we obtain the statement of Lemma \ref{bUissb}.
	\end{proof}
	
	\begin{lemma} \label{semigroup}
		$\calSB(\mN_0)$ is a closed semigroup.
	\end{lemma}
	\begin{proof}
		Let $A,B \in \calSB(\mN_0)$. Consider \begin{equation*}
			\begin{split}
				\sum_{j \in \mN_0}(AB)_{jk} &= \sum_{j \in \mN_0} \sum_{\alpha \in \mN_0} A_{j \alpha} B_{\alpha k } =  \sum_{\alpha \in \mN_0} \sum_{j \in \mN_0} A_{j \alpha} B_{\alpha k }\\
				&\leq \sum_{j \in \mN_0}(AB)_{jk} = \sum_{j \in \mN_0} \sum_{\alpha \in \mN_0} A_{j \alpha} B_{\alpha k } =  \sum_{\alpha \in \mN_0} B_{\alpha k} \leq 1.
			\end{split}
		\end{equation*}
		
		Analogously, we can prove that $\sum_{j \in \mN_0}(AB)_{kj} \leq 1$. Suppose that $B_n \to B$, i.e $ 	\lim_{n\to \infty}\|B_n - B \|_{\l^1(\mN_0)}=0.$ Therefore, we have the weak convergence, i.e. for all $f \in \big(\l^{1}(\mN_0)\big)^*\cong l^{\infty}(\mN_0)$ we have \begin{equation}\label{limitequal0}
			\lim_{n\to \infty}f\big((B_n - B)x \big)=0.
		\end{equation}
		
		Consider the following linear functional $f \in \big(\l^{1}(\mN_0)\big)^*$ \begin{equation*}
			y=(1,1,\ldots) \in l^{\infty} (\mN_0),\quad f_y(x) = (x,y) = \sum_{k \in \mN_0} x_k.
		\end{equation*}

		Using (\ref{limitequal0}), we get \begin{gather*}
			\big| f_y\big((B_n - B)e_{\alpha} \big) \big| = \Big|\sum_{j\in\mN_0} (\bB_{n})_{j\alpha} - \sum_{j\in\mN_0}\bB_{j\alpha}\Big| < \varepsilon,\\
			-\varepsilon	\leq-\varepsilon+\sum_{j\in\mN_0} (\bB_{n})_{j\alpha}< \sum_{j\in\mN_0}\bB_{j\alpha} < \varepsilon+\sum_{j\in\mN_0} (\bB_{n})_{j\alpha}\leq \varepsilon+1.
		\end{gather*}
		
		Then $	0\leq\sum_{j\in\mN_0}\bB_{j\alpha} \leq 1.$ Consider the sequence $\{B_n ^T\colon (B_n ^{T})_{jk} =  (B_n)_{kj}\}_{n \in \mN_0}$. Suppose \begin{equation*}
			\lim_{n\to \infty} \|B_n ^T - B' \|_{l^1 (\mN_0)} =0.
		\end{equation*}
		
		Consider the lineal functional $f_k \in \big(\l^{1}(\mN_0)\big)^*$ defined by $	f_{k} (x) = x_k.$ Hence, we obtain \begin{equation*}
			\lim_{n\to \infty}f_k(B^T _n e_j) = \lim_{n\to \infty} (B^T _{n})_{kj} = (B')_{kj}.
		\end{equation*}
		
		On the other hand, we have \begin{equation*}
			\lim_{n\to \infty} (B^T _{n})_{kj} = \lim_{n\to \infty} (B _{n})_{jk} = B_{jk}.
		\end{equation*}
		
		Therefore $B' = B^T$ and $B_n ^T \to B^T$. Thus, we can similarly obtain \begin{equation*}
			0\leq\sum_{j\in\mN_0}\bB_{\alpha j} \leq 1.
		\end{equation*}

	\end{proof}

	\subsection{Sequence $\{a_{\alpha} ^k (U)\}_{k \in \mN}$.}
	
	The sequence below will be useful for calculating entropy. Suppose $U \in \calCU_{\mu}(\calH),\; b(U)=\bB \in \calSB(\mN_0)$. Consider the sequence $\{a_{\alpha} ^k (U)\}_{k \in \mN} \colon$ 
	\begin{equation} \label{sequenceA}
		\alpha \in \mN_0, \qquad a_{\alpha} ^k (U) = \sum_{j_{k-1},\ldots,j_0 \in \mN_0} \bB_{j_{k-1} j_{k-2}} \bB_{j_{k-2} j_{k-3}} \ldots \bB_{j_0 \alpha}. 
	\end{equation}  
	
	\begin{lemma}\label{propertiesA}
		Suppose $U \in \calCU_{\mu}(\calH),\; b(U)=\bB \in \calSB(\mN_0)$. Then 
		\begin{equation*}
			a_{\alpha}^k (U) \geq 0, \quad  a_{\alpha}^k (U) \leq 1,\quad a_{\alpha}^{k+1} (U) \leq a_{\alpha}^{k} (U). 
		\end{equation*}
	\end{lemma}
	\begin{proof}
		Using equation (\ref{sequenceA}) and definition of $\bB \in \calSB(J)$, we get 
		\begin{equation*}
			\begin{split}
				0 &\leq a_{\alpha}^{k+1} (U) = \sum_{j_{k},\ldots,j_0 \in \mN_0} \bB_{j_{k} j_{k-1}} \bB_{j_{k-1} j_{k-2}} \ldots \bB_{j_0 \alpha}	\\
				&\leq \sum_{j_{k-1},\ldots,j_0\in \mN_0} \bB_{j_{k-1} j_{k-2}}  \ldots \bB_{j_0 \alpha} =a_{\alpha}^{k}(U) \leq 1.
			\end{split}
		\end{equation*}
	\end{proof}

\begin{corollary}
	For any $U \in \calCU_{\mu}(\calH)$ and for any $\alpha \in \mN_0$ the sequence $\{a_{\alpha} ^k\}_{k\in \mN_0}$ converges.
\end{corollary}

	Suppose $U \in \calCU_{\mu}(\calH),\; b(U)=\bB \in \calSB(\mN_0)$. Consider the sequence $\{u_{\alpha}^n (U)\}_{n\in\mN_0}$: \begin{equation} \label{sequenceUalpha}
		\alpha \in \mN_0 ,\quad u_{\alpha}^n (U) = \sum_{k=1}^n a_{\alpha}^k (U).
	\end{equation}
	
	By definition, we put 
	\begin{equation} \label{numbersUalpha}
		u_{\alpha}(U) = \lim_{n \to \infty} \frac{1}{n} u_{\alpha}^n (U).
	\end{equation}

	\begin{corollary} \label{lemmaaboutualpha}
		Suppose $U \in \calCU_{\mu}(\calH),\; b(U)=\bB \in \calSB(\mN_0)$.  Then 
		\begin{equation}\label{inequU}
			u_{\alpha}(U)= \lim_{n \to \infty} a_{\alpha}^n (U), \quad u_{\alpha}(U) \leq 1.
		\end{equation} 
	\end{corollary}

	\subsection{Operator $\bB_{n,\alpha}$}
	
	Suppose $T\in \calL(E)$, $E$ a Banach space. We define $A_n = \frac{1}{n} \sum_{k=0}^{n-1} T^k$.
	
		\begin{theorem} \label{theoremforl1}
		 Suppose that $\sup_{n \in \mN_0} \|A_n \| < \infty$
		and $(1/n)T^n f \to 0$ for all $f \in E$. Then the subspace \begin{equation*}
			F = \{f\colon\lim_{n \to \infty} A_n f \text{ exists} \}
		\end{equation*}
		is closed, $T$-invariant, and decomposed into a direct sum of closed subspaces \begin{equation*}
			F = \Ker(T-I) \oplus \overline{\text{Im} (T-I)}.
		\end{equation*}
	\end{theorem}
	\begin{proof}
		Theorem was proved in \cite{Eisner}.
	\end{proof}

	\begin{lemma} \label{oplusofl1}
		Suppose $\bB \in \calSB(\mN_0)$. Then \begin{equation*}
			l^1(\mN_0) = \text{Ker}(\bB - I) \oplus \overline{\text{Im}(\bB -I)}.
		\end{equation*}
	\end{lemma}
	\begin{proof}
		By definition, put \begin{equation*}
			\bB_n = \sum_{j=0}^{n-1} \bB^j.
		\end{equation*}
		
		According to Theorem \ref{theoremforl1} we consider the subspace $F = \{x\in l^1(\mN_0)\colon  \exists\;  \lim_{n\to \infty} \frac{1}{n}\bB_n x   \}$. Suppose $x \in l^1(\mN_0)$. We denote \begin{equation*}
			s_n(x) = \|\bB_{n} x\|_{l^1(\mN_0)}.
		\end{equation*}
	
	Since \begin{equation*}
		\begin{split}
		s_{n+m}(x) &= \| \bB_{n+m} x \|_{l^1(\mN_0)} =\bigg\| \sum_{j=0}^{n+m-1} \bB^j x\bigg\|_{l^1(\mN_0)} \\
		&\leq \bigg\| \sum_{j=0}^{m-1} \bB^{j} x\bigg\|_{l^1(\mN_0)} +\bigg\|\bB^m \sum_{j=0}^{n-1} \bB^{j} x\bigg\|_{l^1(\mN_0)} \leq s_m(x) + s_n(x).
		\end{split}
	\end{equation*}
		it follows that for any $x \in l^1(\mN_0)$ there exist the limit
		\begin{equation*}
			\lim_{n \to \infty} \frac{s_n(x)}{n} = \lim_{n\to \infty} \frac{\|\bB_n x \|}{n}.
		\end{equation*}

	Therefore, we obtain $F = l^1(\mN_0)$. Using the fact that for any $j\in \mN_0$ $\| \bB^j \|_{l^1(\mN_0)} \leq 1$, we get $\lim_{n \to \infty } \|\frac{1}{n} \bB^n \|_{l^1(\mN_0)} = 0.$ By Theorem \ref{theoremforl1}, we obtain Lemma \ref{oplusofl1}.
	\end{proof}
	
	Suppose $U \in \calCU_{\mu}(\calH),\; b(U)=\bB \in \calSB(\mN_0)$. For any $\alpha \in \mN_0$ and $n \in \mN_0$ consider the operator \begin{equation} \label{aboutoperatorBnalpha}
		\bB_{n,\alpha} = \frac{1}{n} \sum_{j=0}^{n-1} a_{\alpha}^{n-1 -j}(U) \bB^j,
	\end{equation}
where $\{a_{\alpha}^k(U)\}_{k\in \mN_0}$ are defined by (\ref{sequenceA}).
	\begin{lemma} \label{aboutPker}
		Suppose $U \in \calCU_{\mu}(\calH),\; b(U)=\bB \in \calSB(\mN_0)$. Then $\bB_{n,\alpha}$ strongly converges to $u_{\alpha}(U)\calP_{\Ker(\bB -I)}$, i.e. for any $	\forall x \in l^1(\mN_0)\colon$ \begin{equation*}
		 \lim_{n\to \infty}\|\bB_{n,\alpha} x  - u_{\alpha}(U)\calP_{\Ker(\bB -I)} x \|_{l^1(\mN_0)} = 0 .
		\end{equation*}
	\end{lemma}
	\begin{proof}
		Using Lemma \ref{oplusofl1}, we obtain that for all $x \in l^1(\mN_0)$ there exist unique $x_1 \in \Ker(\bB-I)$, $x_2 \in \overline{\text{Im}(\bB-I)}$ such that $	x = x_1 + x_2.$ We put for brevity $ u_{\alpha}=u_{\alpha}(U) $, $ a_{\alpha}^{k}=a_{\alpha}^{k}(U)$. 	We obtain \begin{equation*}
			\bB_{n,\alpha} x_1 = \frac{1}{n} \sum_{k=1}^n a_{\alpha}^{n-k} \bB^{k-1} x_1 = \frac{x_1}{n} \sum_{k=1}^n a_{\alpha}^{n-k} \to u_{\alpha} x_1, \quad (n\to \infty).
		\end{equation*}
		
		There exist $y_2 \neq 0$ such that $x_2 = (\bB -I)y_2$. We get \begin{equation*}
			\begin{split}
				\| \bB_{n,\alpha} x_2 \| &= \| \bB_{n,\alpha} (\bB -I)y_2 \| \\
				&=\frac{1}{n} \Big\|\big(a_{\alpha}^{n-1}\bB^1 - a_{\alpha}^{n-1}\bB^0 + a_{\alpha}^{n-2}\bB^2 -a_{\alpha}^{n-2} \bB^1 +\ldots+ a_{\alpha}^0 \bB^n - a_{\alpha}^0 \bB^{n-1}\big) y_2\Big\| \\
				&= \frac{1}{n} \Big\|-a_{\alpha}^{n-1} \bB^0y_2+\big(a_{\alpha}^{n-1} - a_{\alpha}^{n-2}\big) \bB^1 y_2 + \big(a_{\alpha}^{n-2} - a_{\alpha}^{n-3} \big) \bB^2 y_2 + \ldots  \\
				&\qquad\quad \qquad + \big(a_{\alpha}^1 - a_{\alpha}^0) \big) \bB^n y_2 + a_{\alpha}^n \bB^n y_2 \Big \|  \\
				&\leq\frac{\|y_2\|}{n}\Big( \big|a_{\alpha}^{n-1} - a_{\alpha}^{n-2}\big| + \big|a_{\alpha}^{n-2} - a_{\alpha}^{n-3}\big| +\ldots+\big|a_{\alpha}^1 - a_{\alpha}^0 \big| + a_{\alpha}^0 + a_{\alpha}^{n-1}\Big)
			\end{split}
		\end{equation*}
		
		Using Lemma \ref{propertiesA}, we get \begin{equation*}
			\begin{split}
				\| \bB_{n,\alpha} x_2 \|  \leq&\frac{\|y_2\|}{n}\Big( \big(a_{\alpha}^{n-1} - a_{\alpha}^{n-2}\big) + \big(a_{\alpha}^{n-2} - a_{\alpha}^{n-3}\big) + \ldots \\
				&\qquad \qquad \qquad \ldots+\big(a_{\alpha}^1 - a_{\alpha}^0 \big) + a_{\alpha}^0 + a_{\alpha}^{n-1}\Big)= \frac{2\|y_2\|}{n} a_{\alpha}^{n-1}. 	
			\end{split}
		\end{equation*}
		
		Therefore, we obtain $			\| \bB_{n,\alpha} x_2 \| \to 0 ,\; (n \to \infty).$		
	\end{proof}
	
	\section{Definition of entropy} \label{definitionoftheentropy}
	
	\begin{lemma} \label{aboutnu}
		Let $U \in \calCU_{\mu}(\calH)$, $\bB = b(U)$. Then for any diagonal operators $\wh{g_1},\wh{g_2}$ \begin{equation*}
			\|\wh{g_2}U\wh{g_1} \|_{\mu} ^2 = \sum_{j,k \in \mN_0} |(\wh{g_2})_j|^2 \bB_{jk} |(\wh{g_1})_k|^2 \mu_k.
		\end{equation*}
	\end{lemma}
	
	\begin{proof}
		If we combine Lemma \ref{g1Wg2}, Lemma \ref{boundB}, we get Lemma \ref{aboutnu}. 
	\end{proof}
	
	For any $U \in \calCU(\cal{H})$ and the ordered collection of diagonal operators $\bG = \{g_0,\ldots,g_n\}, \quad n\geq 1,$ we put \begin{equation} \label{IwG}
		\calI_{U}(\bG) = \sum_{j_0,\ldots,j_n \in \mN_0} |(\wh{g_n})_{j_n}|^2 \bB_{j_n j_{n-1}} |(\wh{g_{n-1}})_{j_{n-1}}|^2 \ldots \bB_{j_1j_0}|(\wh{g_0})_{j_0}|^2 \mu_{j_0},
	\end{equation} 
	where $ \nu = b(U)$.
	
	\begin{remark}
		The series (\ref{IwG}) converges.
	\end{remark}
	\begin{proof}
		Since all terms in the sum (\ref{IwG}) are no negative, it is sufficient to prove that $\calI_U(\bG)$ is bounded from above. It is clear that \begin{equation*}
			\begin{split}
				\calI_U(\bG) &\leq \| \wh{g_n}\|_{\infty} \| \wh{g_{n-1}}\|_{\infty} \ldots \| \wh{g_0}\|_{\infty}\sum_{j_0,\ldots,j_n \in \mN_0}  \bB_{j_n j_{n-1}}  \ldots \bB_{j_1j_0} \mu_{j_0} \\
				& \leq \|\wh{g_n}\|_{\infty} \| \wh{g_{n-1}}\|_{\infty} \ldots \| \wh{g_0}\|_{\infty}, \qquad \| x \|_{\infty} = \sup_{j\in \mN_0} |x_j|.
			\end{split}
		\end{equation*}
	\end{proof}
	
	\begin{lemma}
		Let $W \in \calL(\calH)$. Then the following assertions hold:
		
		$\bullet$ if $n=1$, then $\calI_{W}(\{\widehat{g}_0, \widehat{g}_1\}) = \|\widehat{g}_1 W \widehat{g}_0 \|^2 _{\mu}$,
		
		$\bullet$ let $W = U_F$, where $F \in \text{Aut}(\mN_0,\mu)$ and let $\bG = \{\widehat{g}_0,\widehat{g}_1,\ldots,\widehat{g}_n\}$. Then \begin{equation*}
			\calI_{U_F}(\bG) = \|\widehat{g}_n U_F \widehat{g}_{n-1} \ldots U_F \widehat{g}_0 \|_{\mu} ^2.
		\end{equation*}
	\end{lemma}
	\begin{proof}
		The first assertion follows from Lemma \ref{equog_1Wg_2}. Let us prove the second assertion. Using equations (\ref{koopmanoperator}) and (\ref{themapbdef}), we obtain $			\big(b(U_F)\big)_{jk} = \delta_{F(j)k},$ where $j,k \in \mN_0.$	Thus we have \begin{equation*}
			\begin{split}
				\calI_{U_F} (\bG) &= \sum_{j_0,\ldots,j_n \in \mN_0} |(\widehat{g}_n)_{j_n}|^2 \ldots |(\widehat{g}_0)_{j_0}|^2 \delta_{F(j_n) j_{n-1}} \ldots \delta_{F(j_1) j_0} \mu_{j_0}\\
				&=\sum_{j_n\in \mN_0} |(\widehat{g}_n)_{j_n}|^2 |(\widehat{g}_{n-1})_{F(j_n)}|^2 \ldots |(\widehat{g}_0)_{F^n (j_n)}|^2.
			\end{split}
		\end{equation*}
		
		Using Lemma \ref{lemconnecttwodef}, we obtain the second assertion.
		
	\end{proof}

	Let $\calS_{n,k}$ be the set of all maps $\{0,\ldots,n\} \to \{0,\ldots,k\}$. For any partition $\chi = \{X_0,\ldots,X_{J}\}$ of $\mN_0$ and for any $\sigma \in \calS_{n,J}$ we put
	\begin{equation*}
		\bG_{\sigma} = \bG_{\sigma}(\chi) = \{\one_{X_{\sigma(0)}},\ldots,\one_{X_{\sigma(n)}}\}.
	\end{equation*} 
	
	\begin{lemma} \label{calcI}
		Suppose $U \in \calCU_{\mu}(\calH)$ and $\chi=\{X_0,\ldots,X_K\}$ is a partition of $\mN_0$, $\sigma \in \calS_{n,k}$. Then \begin{equation} \label{eqIu}
			\calI_{U}(\bG_{\sigma}) = \sum_{j_0,\ldots,j_n \in \mN_0} \mu_{j_n} |U_{j_n j_{n-1}} \ldots U_{j_1 j_0}|^2 \one_{X_{\sigma(n)}} (j_n) \ldots \one_{X_{\sigma(0)}} (j_0).
		\end{equation} 
	\end{lemma}
	\begin{proof}
		Using (\ref{IwG}) and $\nu = b(U)$, we get (\ref{eqIu}).
	\end{proof}
	
	For any $U \in \calCU_{\mu}(\calH)$ and any $n=1,2\ldots$ we define
	\begin{equation} \label{defmfhUchin}
		\mathfrak{h}(U,\chi,n) = - \sum_{\sigma \in \calS_{n,J}} \calI_{U}(\bG_{\sigma}) \log \calI_{U}(\bG_{\sigma}).
	\end{equation}
	
	If there exist the limit \begin{equation}\label{limitofentr}
		\mathfrak{h}(U,\chi) = \lim_{n \to \infty} \frac{1}{n} \mathfrak{h}(U,\chi,n).
	\end{equation} 
	then we define entropy of $U \in \calCU_{\mu}(\calH)$ by 
	\begin{equation*}
		\mfh(U) = \sup_{\chi} \mfh(U,\chi).
	\end{equation*}
	
	This construction was proposed in \cite{AfonTres22}. In \cite{AfonTres22}, it was shown that, for regular (see \cite{Tres21}) unitary operators
	on $L^2(\mT^n, \mu)$, where $\mT^n$ is the torus and $\mu$ is the Lebesgue measure on $\mT^n$ $(d\,\mu = (1/(2\pi)^n) d\,x_1 \ldots d\,x_n)$, the
	limit (\ref{limitofentr}) exists for Koopman operators on $L^2(\calX , \mu)$ and for finite-dimensional operators on $\mC^J$ with
	the “uniform” measure. In addition, it was shown in \cite{AfonTres22,TresCher22} that, on these operators, the function $\mfh(U, \chi)$
	approaches its upper bound (\ref{limitofentr}) as the partition $\calX$ is refined.
	
	\section{Monotonicity of $\mfh$} \label{monotonicityofmfh}
	
	Before proving the main assertion of this section, we prove two auxiliary lemmas. Consider a mapping $p\colon \mathbb{Z}_{k+1} \to \mathbb{Z}_{k}$ such that
	\begin{equation*}
		p(j)=j\quad \text{if} \quad j\in\{0,1,\ldots,k \}\quad \text{and}\quad p(k+1)=k. 
	\end{equation*}

	\begin{lemma} \label{aboutp}
		Let $\sigma \in \mathcal{S}_{n,k}$, $A_{\sigma}=\{\lambda \in \mathcal{S}_{n,k+1}\; |\; p \circ \lambda = \sigma\}$. Then
		\begin{equation*}
			\mathcal{S}_{n,k+1} = \bigsqcup_{\sigma \in \mathcal{S}_{n,k}} A_{\sigma}.
		\end{equation*}
	\end{lemma}
	
	\begin{proof}
		Let $	\lambda \in A_{\sigma'} \cap A_{\sigma''} \text{ for } \sigma' \neq \sigma''.$
		Then $\sigma' = p \circ \lambda = \sigma''.$ Therefore, $A_{\sigma'} \cap A_{\sigma''} = \varnothing.$
		
		Let us prove the following inclusion: \begin{equation*}
			\mathcal{S}_{n,k+1} \subset \bigcup_{\sigma \in \mathcal{S}_{n,k}} A_{\sigma}.
		\end{equation*} 
		
		Let $\lambda \in \mathcal{S}_{n,k+1}$, $ M = \{m \in \{0,\ldots,n\}\;|\;\lambda(m)=k+1\}.$ By definition, we put $\sigma \in \calS_{n,k}$ so that \begin{equation*}
			\sigma(m)=\begin{cases}
				k,\quad m\in M\\
				\lambda(m),\qquad m \in \{0,\ldots,n\} \setminus M. 
			\end{cases}
		\end{equation*} 
	
	Evidently, for any $m\in \{0,\ldots,n\}\colon$ $p\circ \lambda(m) = \sigma(m)$. The inverse inclusion
		\begin{equation*}
			\mathcal{S}_{n,k+1} \supset \bigcup_{\sigma \in \mathcal{S}_{n,k}} A_{\sigma}
		\end{equation*}  holds by the definition of the set $A_{\sigma}$. Lemma is proved.
	\end{proof}
	
	Let $\kappa =\{Y_0,\ldots, Y_{k+1}\}$ be a subpartition of the partition $\chi = \{X_0,\ldots,X_k\}$ such that $X_0=Y_0, \ldots, X_{k-1}=Y_{k-1},\; X_k = Y_k \cup Y_{k+1}$. Consider the collection $\mathbf{G}_{\lambda}(\kappa) = \{\widehat{\mathbf{1}}_{Y_{\lambda(0)}},\ldots, \widehat{\mathbf{1}}_{Y_{\lambda(n)}}\},\; \lambda \in \mathcal{S}_{n,k+1}.$
	
	\begin{lemma}\label{aboutsumlambda}
		Let $U \in \calCU_{\mu}(\calH)$. Then \begin{equation*}
			\sum_{\lambda \in A_{\sigma}} \mathcal{I}_{U}(\mathbf{G}_{\lambda}(\kappa)) = \mathcal{I}_{U}(\mathbf{G}_{\sigma}(\chi)).
		\end{equation*}
	\end{lemma}
	\begin{proof}
		Let $M = \{m \in \mathbb{Z}_n \colon \sigma(m) \in \{0,\ldots,k-1\}\}$. Then
		\begin{equation*}
			\begin{split}
				\mathcal{I}_{U}(\mathbf{G}_{\sigma}(\chi)) &=  \sum_{j_0,\ldots,j_n \in \mN_0} \mu_{j_n} |U_{j_n j_{n-1}} \ldots U_{j_1 j_{0}}|^2 \mathbf{1}_{X_{\sigma(n)}} (j_n) \ldots \mathbf{1}_{X_{\sigma(0)}} (j_0) \\
				& =  \sum_{j_0,\ldots,j_n \in \mN_0} \mu_{j_n} |U_{j_n j_{n-1}} \ldots U_{j_1 j_{0}}|^2 \prod_{m \in M}\mathbf{1}_{Y_{\sigma(m)}} (j_m) \prod_{m \in \mathbb{Z}_J \setminus M}\mathbf{1}_{X_k} (j_m)\\
				& =  \sum_{j_0,\ldots,j_n \in \mN_0} \mu_{j_n} |U_{j_n j_{n-1}} \ldots U_{j_1 j_{0}}|^2 \prod_{m \in M}\mathbf{1}_{Y_{\sigma(m)}} (j_m) \sum_{\lambda \in A_{\sigma}}\prod_{m \in \mathbb{Z}_J \setminus M}\mathbf{1}_{Y_{\lambda(m)}} (j_m)\\
				&= \sum_{\lambda \in A_{\sigma}} \mathcal{I}_{U}(\mathbf{G}_{\lambda}(\kappa)).
			\end{split}
		\end{equation*} 
		Lemma is proved.
	\end{proof}
	
	\begin{corollary}\label{aboutsumlambda1}
		Let $U \in \calCU_{\mu}(\calH)$. Then \begin{equation*}
			\mathcal{I}_{U}(\mathbf{G}_{\lambda}(\kappa)) \leq \mathcal{I}_{U}(\mathbf{G}_{\sigma}(\chi)) \; \textrm{for any} \; \lambda \in A_{\sigma}.
		\end{equation*}
	\end{corollary}

	\begin{lemma} \label{monot}
		Let $U \in \mathcal{U}(J,\mu)$. Let $\kappa$ be a subpartition of the partition $\chi$. Then $\mathfrak{h}(U,\chi,n) \leq \mathfrak{h}(U,\kappa,n)$.
	\end{lemma}
	\begin{proof}
		Without loss of generality, we will assume that the partition $\kappa =\{Y_0,\ldots, Y_{k+1}\}$ is a subpartition of the partition $\chi = \{X_0,\ldots,X_k\}$ such that $X_0=Y_0, \ldots, X_{k-1}=Y_{k-1},\; X_k = Y_k \cup Y_{k+1}$. Then, using Lemma \ref{aboutp}, we obtain
		\begin{equation*}
			\mathfrak{h}(U,\kappa,n) - \mathfrak{h}(U,\chi,n) = \sum_{\sigma \in \mathcal{S}_{n,K}} \Delta_\sigma,
		\end{equation*}
		\begin{equation*}
			\Delta_{\sigma} =\mathcal{I}_{U}(\mathbf{G}_{\sigma}(\chi)) \log \mathcal{I}_{U}(\mathbf{G}_{\sigma}(\chi)) - \sum_{\lambda \in A_{\sigma}}\mathcal{I}_{U}(\mathbf{G}_{\lambda}(\kappa)) \log \mathcal{I}_{U}(\mathbf{G}_{\lambda}(\kappa)).
		\end{equation*}
		
		By Lemma \ref{aboutsumlambda} and Corollary  \ref{aboutsumlambda1}, we have \begin{equation*}
			\begin{split}
				\Delta_{\sigma} &=\sum_{\lambda \in A_{\sigma}}\mathcal{I}_{U}(\mathbf{G}_{\lambda}(\kappa)) \log \mathcal{I}_{U}(\mathbf{G}_{\sigma}(\chi)) - \sum_{\lambda \in A_{\sigma}}\mathcal{I}_{U}(\mathbf{G}_{\lambda}(\kappa)) \log \mathcal{I}_{U}(\mathbf{G}_{\lambda}(\kappa)) \\ 
				&=\sum_{\lambda \in A_{\sigma}}\mathcal{I}_{U}(\mathbf{G}_{\lambda}(\kappa))\log \frac{\mathcal{I}_{U}(\mathbf{G}_{\sigma}(\chi))}{\mathcal{I}_{U}(\mathbf{G}_{\lambda}(\kappa))} \geq 0.
			\end{split}
		\end{equation*}
	\end{proof}

	\section{Calculation of entropy} \label{Calcofentropy}
	
		Let $\mC^J$ be Hilbert space with the scalar product and norm \begin{equation*} 
		\langle x, y \rangle = \sum_{j=0}^{J-1} \mu_j x_j \overline{y_j},\qquad \|x \| = \sqrt{\langle x,x\rangle}.
	\end{equation*}
	
	By definition, we put \begin{equation} \label{deffinitecu}
		\calCU_{\mu}(\mC^J) = \{U \in \calL(\mC^J)\colon\ \| Ux\| \leq \| x\|\}.
	\end{equation}

	Suppose $U \in \calCU_{\mu}(\mC^J)$. Let $(U_{jk})_{j,k \in \mZ_J}$ be the elements of matrix $U$.

	Using Lemma \ref{oplusofl1}, we have \begin{equation*}
		\forall\; x \in \mC^J \;\exists! \; x_1 \in  \text{Ker}(\bB - I),\; x_2 \in \text{Im} (\bB -I)\colon x=x_1+x_2.
	\end{equation*}
	
	Now we introduce the projector $\calP_{\text{Ker}(\bB -I)}$ on $\text{Ker}(\bB -I)$ such that 
	\begin{equation} \label{projectoronKer}
		\calP_{\Ker(\bB -I)} \colon x \mapsto x_1.
	\end{equation}

	\begin{theorem} \label{entropyoffinitedim}
		Suppose $U \in \calCU_{\mu}(\mC^J)$, $\mu = (\mu_0,\ldots,\mu_{J-1})$, $e = (1,\ldots,1)$. Then \begin{equation*}
			\mfh(U) = - \sum_{j,k=0 }^{J-1} \big(\calP^T_{\text{Ker}(\bB-I)} e\big)_j \bB_{jk} \big(\calP_{\text{Ker}(\bB-I)} \mu\big)_k  \log \bB_{jk},\quad b(U)=\bB.
		\end{equation*}
	\end{theorem}
	\begin{proof}
		
	Consider the following partition of $\{0,\ldots,J-1\}$ \begin{equation*}
			\chi_{\odot} = \{\{0\},\{1\},\ldots, \{J-1\}\}.
		\end{equation*}
		
		Let $\calS_{n,J-1}$ be the set of maps $\{0,\ldots,n\}\to \{0,\ldots,J-1\}$. Now, by Lemma \ref{calcI}, \begin{equation} \label{Ionfinest}
			\calI_{U}(\bG_{\sigma})=\calI_{U}(\bG_{\sigma}(\chi_{\odot})) = \bB_{\sigma(n) \sigma(n-1)}\ldots \bB_{\sigma(1)\sigma(0)} \mu_{\sigma(0)},\qquad \sigma \in \calS_{n,J-1}.
		\end{equation}
		
		Using Lemma \ref{monot} and (\ref{Ionfinest}), we calculate the entropy $\mfh(U)$ on the finest partition $\chi_{\odot}$ \begin{equation*}
			\begin{split}
				-\mfh(U,\chi_{\odot},n) &= \sum_{\sigma \in \calS_{n,J-1}} \calI_U (\bG_{\sigma}) \log \calI_U (\bG_{\sigma}) \\ 
				&=  \sum_{k=1}^n \sum_{j_0,\ldots,j_n =0}^{J-1}  \bB_{j_n j_{n-1}} \ldots \bB_{j_1 j_0} \mu_{j_0} \log \bB_{j_k j_{k-1}}+\\
				&\qquad\qquad+ \sum_{j_n,\ldots j_0 =0}^{J-1}  \bB_{j_n j_{n-1}} \ldots \bB_{j_1 j_0} \mu_{j_0} \log \mu_{j_0}  = \sum_{k=1}^n S_k,
			\end{split}
		\end{equation*}
		where \begin{equation*}
			S_k =  \sum_{j_0,\ldots,j_n=0}^{J-1}  \bB_{j_n j_{n-1}} \ldots \bB_{j_1 j_0} \mu_{j_0} \log \bB_{j_k j_{k-1}},\quad k=1,\ldots,n,
		\end{equation*}
		\begin{equation*}
			S_0 = \sum_{j_0,\ldots,j_n =0}^{J-1}  \bB_{j_n j_{n-1}} \ldots \bB_{j_1 j_0} \mu_{j_0} \log \mu_{j_0}.
		\end{equation*}
		
		We obtain \begin{equation*}
			\begin{split}
				S_k &= \sum_{j_0,\ldots,j_n =0}^{J-1} \bB_{j_n j_{n-1}} \ldots\bB_{j_{k+1} j_k} \bB_{j_k j_{k-1}}\bB_{j_{k-1} j_{k-2}} \ldots \bB_{j_1 j_0} \mu_{j_0} \log \bB_{j_k j_{k-1}}\\
				&= \sum_{\alpha,\beta,\gamma =0}^{J-1}  a_{\alpha}^{n-k}(U) \bB_{\alpha \beta} \bB_{\beta \gamma}^{k-1}  \mu_{\gamma} \log \bB_{\alpha \beta}, \text{ where $a_{\alpha}^{n-k}(U)$ are defined by (\ref{sequenceA}) }.
			\end{split}
		\end{equation*}
		
		Thus, we have \begin{equation*}
			\begin{split}
				\frac{1}{n}\sum_{k=1}^n S_k &= \sum_{\alpha,\beta,\gamma =0}^{J-1} \bB_{\alpha \beta} \Bigg(\frac{1}{n} \sum_{k=1}^n a_{\alpha}^{n-k}(U) \bB^{k-1} \Bigg)_{\beta \gamma} \mu_{\gamma} \log \bB_{\alpha\beta} \\
				& = \sum_{\alpha,\beta,\gamma =0}^{J-1} \bB_{\alpha \beta} (\bB_{n,\alpha})_{\beta \gamma} \mu_{\gamma} \log \bB_{\alpha\beta}, \text{ where $\bB_{n,\alpha}$ are defined by (\ref{aboutoperatorBnalpha}) }.
			\end{split}
		\end{equation*}
		
		By Lemma \ref{aboutPker}, so that 
		\begin{equation*}
			\begin{split}
				\lim_{n \to \infty} \frac{1}{n} \sum_{k=1}^n S_k &= \sum_{\alpha,\beta,\gamma =0}^{J-1} u_{\alpha}(U) \bB_{\alpha \beta} \big(\calP_{\Ker(\bB -I)} \big)_{\beta\gamma} \mu_{\gamma} \log \bB_{\alpha \beta} \\ 
				&=\sum_{\alpha,\beta,\gamma =0}^{J-1} u_{\alpha}(U) \bB_{\alpha \beta}\big(\calP_{\Ker(\bB -I)}\mu \big)_{\beta} \log \bB_{\alpha \beta}.
			\end{split}
		\end{equation*}
	
	By using (\ref{numbersUalpha}) and Lemma \ref{aboutPker}, we obtain \begin{equation*}
		\begin{split}
		u_{\alpha}(U) &= \lim_{n\to \infty} \frac{1}{n} \sum_{k=1}^n \sum_{j_{k-1},\ldots, j_0 =0}^{J-1} \bB_{j_{k-1} j_{k-2}} \bB_{j_{k-2} j_{k-3}} \ldots \bB_{j_{0} \alpha} = \lim_{n\to \infty} \frac{1}{n} \sum_{k=1}^n \sum_{j=0}^{J-1} \big(\bB^{k-1}\big)_{j \alpha} \\
		&=\sum_{j=0}^{k-1} \bigg( \lim_{n\to \infty} \frac{1}{n} \sum_{k=1}^n \bB^{k-1} \bigg)_{j \alpha} = \sum_{j=0}^{k-1} \big( \calP_{\text{Ker}(\bB-I)}\big)_{j \alpha} = \big( \calP^T_{\text{Ker}(\bB-I)} e\big)_{\alpha}.
		\end{split}
	\end{equation*}
		
		Finally, we obtain \begin{equation*}
			\frac{1}{n} S_0 = \frac{1}{n}\sum_{\alpha =0}^{J-1}  a_{\alpha}^{n}(U)   \mu_{\alpha} \log\mu_{\alpha} \to 0,\quad n \to \infty.
		\end{equation*}
	\end{proof}

	\section{Properties of $\mfh$} \label{propmfh}
\subsection{Multiplication by diagonal operator}

\begin{lemma} \label{entropyd1Ud2}
	Suppose $U \in \calCU_{\mu}(\mC^J)$, $\calD_1, \calD_2 \in \calU_{\mu}(\mC^J)$ are diagonal operators. Then  \begin{equation*}
		\mfh(\calD_1 U \calD_2) = \mfh(U).
	\end{equation*}
\end{lemma}
\begin{proof}
	We have \begin{gather*}
		\calD_1 = \text{diag} (a_0,\ldots,a_{J-1}),\qquad |a_j|= 1,\quad j \in \{0,\ldots,J-1\},  \\
		\calD_2 = \text{diag} (b_0,\ldots,b_{J-1}),\qquad  |b_j|= 1,\quad j \in \{0,\ldots,J-1\}.
	\end{gather*}
	
	Thus we have \begin{equation*}
		(\calD_1 U \calD_2)_{jk} = a_j U_{jk} b_k, \qquad b(\calD_1 U \calD_2) =\frac{\mu_j}{\mu_k} |a_j|^2 |U_{jk}|^2 |b_k|^2 = \frac{\mu_j}{\mu_k} |U_{jk}|^2 = b(U). 
	\end{equation*}
	
	The result is \begin{equation*}
		j \in \{0,\ldots,J-1\},\quad u_j(\calD_1 U \calD_2) = u_j(U),\qquad v_j(\calD_1 U \calD_2) = v_j(U). 
	\end{equation*}
	
	This completes the proof of Lemma \ref{entropyd1Ud2}.
\end{proof}

\subsection{Conjugation by Koopman operator}

\begin{lemma}
	Suppose $U \in \calCU(\calH)$. Let $F\colon \{0,\ldots,J-1\} \to \{0,\ldots,J-1\}$ be an authomorphism. Then \begin{equation*}
		\mfh(U_F ^{-1} U U_F) = \mfh(U).
	\end{equation*} 
\end{lemma}
\begin{proof}
	Using (\ref{koopmanoperator}), we get \begin{equation*}
		(U^{-1}_F)_{jk} = \sqrt{\frac{\mu_{F^{-1}(j)}}{\mu_j}} \delta_{F^{-1}(j)k},\qquad j,k \in \{0,\ldots,J-1\}.
	\end{equation*}
	
	Therefore, we have \begin{equation*}
		\begin{split}
			\big(U^{-1}_F U U_F\big)_{jk}& = \sum_{\alpha_1,\alpha_2 =0}^{J-1} (U^{-1}_F)_{j\alpha_1} U_{\alpha_1 \alpha_2} (U_F)_{\alpha_2 k}\\
			&= \sum_{\alpha_1,\alpha_2 =0}^{J-1} \sqrt{\frac{\mu_{F^{-1}(j)}}{\mu_j}} \delta_{F^{-1}(j)\alpha_1} U_{\alpha_1 \alpha_2} \sqrt{\frac{\mu_{k}}{\mu_{\alpha_2}}}\delta_{F(\alpha_2)k}\\
			&=\sqrt{\frac{\mu_k}{\mu_j}\frac{\mu_{F^{-1}(j)}}{\mu_{F^{-1}(k)}}}  U_{F^{-1}(j) F^{-1}(k)}, \qquad j,k \in \{0,\ldots,J-1\}.
		\end{split}
	\end{equation*}
	
	If we combine this with definition of the mapping $b$ (see Section \ref{mappingb}). For any $j,k \in \{0,\ldots,J-1\}$ \begin{equation} \label{bufu}
		\big(b(U^{-1}_F U U_F) \big)_{jk} = \frac{\mu_{F^{-1}(j)}}{\mu_{F^{-1}(k)}} |U_{F^{-1}(j) F^{-1}(k)}|^2 = \bB_{F^{-1} (j) F^{-1} (k)}.
	\end{equation}
	
	Using equations (\ref{sequenceA}) and (\ref{bufu}), we get \begin{equation*}
		a_{j} ^k (U^{-1}_F U U_F) = \sum_{j_{k-1},\ldots,j_0=0}^{J-1} \bB_{F^{-1}(j_{k-1}) F^{-1}(j_{k-2})} \ldots \bB_{F^{-1}(j_0) F^{-1}(j)}.
	\end{equation*}
	
	Since $F$ is the authomorphism, it follows that $	a_{j}^k(U^{-1}_F U U_F) = a_{F^{-1}(j)}^k (U).$
	Therefore \begin{equation}\label{ajandajuF}
		u_{j} (U_F ^{-1} U U_F) =\lim_{k \to \infty} 	a_{j}^k(U^{-1}_F U U_F) = \lim_{k\to\infty} a_{F^{-1}(j)}^k (U) = u_{F^{-1}(j)} ( U ).
	\end{equation}
	
	Analogously \begin{equation*}
		\begin{split}
			\big(b^{j}(U^{-1}_F U U_F)\big)_{st} &= (\bB^j)_{F^{-1}(s)F^{-1}(t)} \\
			&= \sum_{\alpha_1,\ldots,\alpha_{n-1}=0}^{J-1} \bB_{F^{-1}(s)F^{-1}(\alpha_1)} \bB_{F^{-1}(\alpha_1) F^{-1}(\alpha_2)} \ldots \bB_{F^{-1}(\alpha_{n-1})F^{-1}(t)}\\
			&=\big(b^{j}(U)\big)_{F^{-1}(s)F^{-1}(t)},\qquad s,t \in \{0,\ldots,J-1\}.
		\end{split}
	\end{equation*}
	
	Thus we have for any $s\in \{0,\ldots,J-1\}$ \begin{equation*} \begin{split}
			v_{s}(U^{-1}_F U U_F) &= \big(\calP_{\text{Ker}(b(U^{-1}_FU U_F) - I)} E_{\mu}\big)_s = \lim_{n \to \infty} \bigg(\frac{1}{n} \sum_{j=0}^{n-1} b^{j} (U^{-1}_F U U_F) E_{\mu} \bigg)_s \\ 
			&=\lim_{n \to \infty} \frac{1}{n} \sum_{j=0}^{n-1} \sum_{k =0}^{J-1}\big(b^{j} (U^{-1}_F U U_F)\big)_{sk} \mu_k=\lim_{n \to \infty} \frac{1}{n} \sum_{j=0}^{n-1} \sum_{k =0}^{J-1}\big(b^{j} ( U )\big)_{F^{-1}(s)F^{-1}(k)} \mu_k \\
			&=\lim_{n \to \infty} \frac{1}{n} \sum_{j=0}^{n-1} \sum_{k=0}^{J-1}\big(b^{j} ( U )\big)_{F^{-1}(s)k} \mu_{F(k)}.
		\end{split}
	\end{equation*}
	
	Using equation (\ref{muusefuleq}), we get \begin{equation} \label{ufveq}
		\begin{split}
			v_{s}(U^{-1}_F U U_F) &= \lim_{n \to \infty} \frac{1}{n} \sum_{j=0}^{n-1} \sum_{k=0}^{J-1}\big(b^{j} ( U )\big)_{F^{-1}(s)k} \mu_k \\
			&=\lim_{n \to \infty} \bigg(\frac{1}{n} \sum_{j=0}^{n-1} b^{j} (U) E_{\mu} \bigg)_{F^{-1}(s)}=v_{F^{-1}(s)}(U),\qquad s \in \{0,\ldots,J-1\}.
		\end{split}
	\end{equation}
	
	Combining Theorem \ref{entropyoffinitedim} and equalities (\ref{ajandajuF}), (\ref{ufveq}), (\ref{bufu}), we obtain \begin{equation*}
		\begin{split}
			\mfh(U_F^{-1}U U_F) &= -\sum_{j,k=0}^{J-1}  u_j (U_F^{-1}U U_F) b_{jk} (U_F^{-1}U U_F) v_{k} (U_F^{-1}U U_F) \log b_{jk} (U_F^{-1}U U_F)\\
			&=-\sum_{j,k=0}^{J-1}   u_{F^{-1}(j)} ( U ) \cdot \bB_{F^{-1}(j)F^{-1}(k)}\cdot v_{{F^{-1}(k)}} (U) \log \bB_{F^{-1}(j)F^{-1}(k)}=\mfh(U).
		\end{split}
	\end{equation*}
	
\end{proof}

\subsection{Operators with zero entropy}
	Let $\{e_0,\ldots,e_{J-1}\}$ be the standard basis of $\mC^J$. Let $f \in (\mC^{J})^*$ be a linear functional on $\mC^{J}$. For any $x\in \mC^J$ we have: \begin{equation*}
		f(x) = \sum_{j=0}^{J-1} f_j x_j.
	\end{equation*}
where $(f_0,\ldots,f_{J-1})$ are coordinates of $f$ in the basis $\{\delta_0,\ldots,\delta_{J-1}\}$ conjugated to $\{e_0,\ldots,e_{J-1}\}\colon$ \begin{equation*}
	\delta_j (x) = x_j,\qquad x \in \mC^J.
\end{equation*}
 Suppose $\bB \in \calSB(\mC^J)$. We define the sets \begin{equation*}
	\mathbf{j}_n (\bB) = \{j_1,\ldots,j_n\}, \qquad 0< j_1<\ldots< j_n< J-1,
\end{equation*}
\begin{equation*}
	\mathbf{k}_m (\bB) = \{k_1,\ldots,k_m\}, \qquad 0< k_1< \ldots < k_m < J-1
\end{equation*} such that for any $j \in \mathbf{j}_n$, $k\in \mathbf{k}_m$ 
\begin{equation*}
e_j(u)=0,\quad \delta_k(v)=0, \text{ for any } u \in \text{Ker}(\bB^T - I),\quad v \in \text{Ker}(\bB - I).
\end{equation*}

We put $\mZ_J = \{0,\ldots,J-1\}.$ We define the set \begin{equation*}
	M = \big\{\bB \in \calSB(\mC^J) \colon \bB_{at},\bB_{tb} \in \{0,1\},\quad a\notin \mathbf{j}_k(\bB),\quad b\notin \mathbf{k}_m(\bB),\quad t \in \mZ_J\big\}.
\end{equation*}

In article \cite{TresCher22}, the semigroup of bistichastic matrices $\calB(\mC^J)$ defined by \begin{equation*}
	\calB(\mC^J) = \bigg\{\bB\in \calSB(\mC^J) \colon \bB_{j,k} \geq 0,\quad \sum_{j=0}^{J-1} \bB_{jk}=1,\quad \sum_{k=0}^{J-1} \bB_{jk}=1,\quad j,k\in \{0,\ldots,J-1\}\bigg\}.
\end{equation*}

	\begin{lemma}
		The following conditions are equivalent:
		\begin{equation*}
		(i)\; \mfh(\bB) = 0;\qquad (ii)\; \bB \in M.
		\end{equation*}
	\end{lemma}
\begin{proof}
	
	\begin{remark}
		In article \cite{TresCher22} we proved that operator $\bB \in \calB(\mC^J)$ with the zero entropy has the form  \begin{equation*}
			\bB_{jk} = \delta_{j \sigma(k)},\quad j,k \in \{0,\ldots,J-1\}, \text{ where}
		\end{equation*}
		$\sigma$ is a bijection of the set $\{0,\ldots,J-1\}$. Note that in particular, this operator belongs to the set $M$.
	\end{remark}
	  
	By using Theorem \ref{entropyoffinitedim}, we obtain for any $j,k \in \{0,\ldots,J-1\}$ \begin{equation*}
		\bB_{jk} = 0, \text{ or } \bB_{jk}=1, \text{ or } \big( \calP^T_{\text{Ker}(\bB-I)}e\big)_j \big( \calP_{\text{Ker}(\bB-I)}\mu\big)_k = 0.
	\end{equation*}

First, if \begin{equation*}
	\big( \calP^T_{\text{Ker}(\bB-I)}e\big)_j = \sum_{s=0}^{J-1} \big(\calP^T_{\text{Ker}(\bB-I)}\big)_{sj} = \sum_{s=0}^{J-1} e_j\big(\calP^T_{\text{Ker}(\bB-I)} \delta_s\big)=0.
\end{equation*}

Note that \begin{equation*}
	\calP^T_{\text{Ker}(\bB-I)} = \lim_{n\to \infty} \frac{1}{n} \sum_{k=0}^{n-1} (\bB^k)^T =\lim_{n\to \infty} \frac{1}{n} \sum_{k=0}^{n-1} (\bB^T)^k=	\calP_{\text{Ker}(\bB^T-I)}.
\end{equation*}

Since for any $j,s \in \{0,\ldots,J-1\}\colon$ $\big(\calP_{\text{Ker}(\bB^T-I)}\big)_{js}=e_j\big(\calP_{\text{Ker}(\bB^T-I)} \delta_s\big) \geq 0$, we obtain for any $s\in \{0,\ldots,J-1\}\colon$ $
e_j\big(\calP_{\text{Ker}(\bB^T-I)} \delta_s\big) = 0$. Since for any $x \in (\mC^J)^*$
	$f_j(\calP_{\text{Ker}(\bB^T-I)}x) = 0$ it follows that for any $u \in \text{Ker}(\bB^T-I)\colon e_j(u)=0$. Hence $j \in \mathbf{j}_k (\bB)$.

Secondly, \begin{equation*}
	\big( \calP_{\text{Ker}(\bB-I)}\mu\big)_k = \sum_{s=0}^{J-1} \big(\calP_{\text{Ker}(\bB-I)}\big)_{ks} \mu_s=0.
\end{equation*}

Since for any $k,s \in \{0,\ldots,J-1\}\colon$ $\big(\calP_{\text{Ker}(\bB-I)}\big)_{ks} \geq 0$ and $\mu_s>0$, we have for any $s \in \{0,\ldots,J-1\}\colon$ $\delta_k (\calP_{\text{Ker}(\bB-I)} e_s)=0.$ Analogously, we may prove that for any $v \in \text{Ker}(\bB-I)\colon \delta_k(v)=0$. Therefore, we get $k \in \mathbf{k}_m$. Thus for $j \notin \mathbf{j}_n$, $k \notin \mathbf{k}_m$ it is necessary $\bB \in \{0,1\}$.  Suppose that $\bB \in M$. Then for any $k\in \mathbf{k}_m$, $v \in \text{Ker}(\bB - I)\colon \delta_k(v) =0.$ In particular for any $s \in \mZ_J$  \begin{equation*}
	0=\delta_k(\calP_{\text{Ker}(\bB - I)}e_s)=\big(\calP_{\text{Ker}(\bB - I)}\big)_{ks}.
\end{equation*} Therefore \begin{equation*}
0=\sum_{s=0}^{J-1} \big(\calP_{\text{Ker}(\bB-I)}\big)_{ks} \mu_s=\big( \calP_{\text{Ker}(\bB-I)}\mu\big)_k.
\end{equation*}

Analogously we may prove that for any $j \in \mathbf{j}_n$, $u \in \text{Ker}(\bB^T -I)$: $\big( \calP^T_{\text{Ker}(\bB-I)}e\big)_j=0$. 
\end{proof}

	\section{Entropy approximation} \label{approximationofentropy}
	\subsection{Sequence $\{U_J\}_{J\in \mN_0}$}
	Suppose $U \in \calCU_{\mu}(\calH)$. Consider the following operators \begin{gather*} \label{operatorsPandQ}
		p_J\colon (x_0,\ldots,x_{J-1},x_J,\ldots) \mapsto (x_0,\ldots,x_{J-1}),\\
		q_J\colon (x_0,\ldots,x_{J-1}) \mapsto (x_0,\ldots,x_{J-1},0,\ldots).
	\end{gather*} There exists a unique operator $U_J\colon \mC^J \to \mC^J$ such that the diagram
	\begin{equation*}
		\begin{tikzcd}
			\mC^J \arrow{d}{q_J} \arrow{r}{U_J}& \mC^J \\
			l^2(\mN_0,\mu)\arrow{r}{U} & l^2(\mN_0,\mu) \arrow{u}{p_J}
		\end{tikzcd}
	\end{equation*}
	is commutative.\begin{lemma}
		Let $U \in \calCU_{\mu}(\calH)$. Then $U_J \in \calCU_{\mu}(\mC^J)$, where the set $\calCU_{\mu}(\mC^J)$ is defined by (\ref{deffinitecu}).
	\end{lemma}
	\begin{proof}
		For any $x=(x_0,x_1,\ldots) \in l^2(\mN_0,\mu)$
		\begin{equation} \label{p_Jx}
			\begin{split}
				\|p_J x \|_{J} &=\|(x_0,\ldots,x_{J-1}) \|_{J} = \|(x_0,\ldots,x_{J-1},0,\ldots) \|_{l^2(\mN_0,\mu)}\\
				&\leq \|(x_0,\ldots,x_{J-1},x_J,\ldots) \|_{l^2(\mN_0,\mu)} = \| x\|_{l^2(\mN_0,\mu)}.
			\end{split}
		\end{equation}
		
		Using (\ref{p_Jx}) and the definition of $U \in \calCU_{\mu}(\calH)$, we obtain \begin{equation*}
			\begin{split}
				\|U_J x \|_{J} &= \|p_J U q_J x \|_{J} \leq  \|U q_J x \|_{l^2(\mN_0,\mu)}\\
				& \leq \|q_J x \|_{l^2(\mN_0,\mu)}= \|(x_0,\ldots,x_{J-1},0,\ldots) \|_{l^2(\mN_0,\mu)} = \| x\|_{J}.
			\end{split}
		\end{equation*}
	\end{proof}
	
	\begin{lemma} \label{strongconv} 	Consider the sequence $\{U_n\}_{n \in \mN_0}$. Then the following assertions hold:

		$\bullet$ The sequence $\{U_n\}_{n \in \mN_0}$ strongly converges to $U$, i.e\begin{equation*}
			 \lim_{n\to \infty}\|U_n x - Ux \|_{l^2(\mN_0,\mu)}= 0 \text{ for any } x \in \calH.
		\end{equation*} 
		
		$\bullet$ The sequence $\{\bB_n^m=b^m(U_n)\}_{n \in \mN_0}$ strongly converges to $\bB^m=b^m(U)$, i.e.
		\begin{equation*}
		\lim_{n\to \infty}\|\bB^m _n x - \bB^m x \|_{l^1(\mN_0)}= 0 \text{ for any } x \in  l^1(\mN_0), \;m \in \mN_0.
		\end{equation*} 
	\end{lemma}
	\begin{proof}
		Let us prove the first assertion. Consider the vector $y = Ux$. Then \begin{equation*}
			y_j = \sum_{s \in \mN_0} U_{js} x_s.
		\end{equation*}
		
		The application of the definition $U \in \calCU_{\mu}(\calH)$ yields \begin{equation} \label{varepsiloninequality}
			|y_j|=\Big|\sum_{s=n}^{\infty} U_{js} x_s\Big| \leq \varepsilon.
		\end{equation}
		
		Using (\ref{varepsiloninequality}) and the definition $U \in \calCU_{\mu}(\calH)$, we obtain \begin{equation*}
			\begin{split}
				\|Ux - U_n x \|_{l^2(\mN_0 ,\mu) }^2 &= \|(U-p_nUq_n)x \|_{l^2(\mN_0 ,\mu) }^2 = \sum_{j \in \mN_0} \mu_j \Bigg|\sum_{s=n}^{\infty} U_{js} x_s\Bigg|^2\\
				&\leq\sum_{j \in \mN_0} \mu_j \varepsilon^2 =\varepsilon^2.
			\end{split}
		\end{equation*}

		Let us prove the second assertion. Using equation (\ref{definitionsemibist}), we have \begin{equation*}
			\|\bB x- \bB_n x \|_{l^1(\mN_0)} = \sum_{j \in \mN_0} \Bigg|\sum_{s=n}^{\infty}\bB_{js} x_s\Bigg| \leq \sum_{j \in \mN_0} \sum_{s=n}^{\infty}|\bB_{js}| |x_s| \leq \sum_{s=n}^{\infty}|x_s| \to 0,\quad n \to \infty.
		\end{equation*}

		Hence, we have the second assertion for $m=1$. Suppose that for $m-1$ we have \begin{equation} \label{assumptioninduction}
			\lim_{n\to \infty} \big\| \bB^{m-1} x -\bB_n ^{m-1} x\big\|_{l^{1} (\mN_0)} = 0.
		\end{equation}
		
		We have the following presentation of $\bB$ \begin{equation*}
			\bB = \bB_n + \bB' _n, \quad \bB_n = b(U_n).
		\end{equation*}
		
		Consider \begin{equation*}
			\begin{split}
				\|\bB^{m} x - \bB^{m} _n x \|_{l^1 (\mN_0)} &= \|\bB^{m-1} \bB x - \bB^{m-1} _n \bB_n x \|_{l^1 (\mN_0)} \\
				&= \|\bB^{m-1}\bB x - \bB^{m-1} _n (\bB-\bB' _n) x \|_{l^1 (\mN_0)} \\
				&\leq  \|(\bB^{m-1} - \bB^{m-1} _n) \bB x \|_{l^1 (\mN_0)} +  \|\bB^{m-1} _n \bB' _n x \|_{l^1 (\mN_0)}.
			\end{split}
		\end{equation*}
		
		Using (\ref{assumptioninduction}), we have \begin{equation*}
			\lim_{n\to \infty} \|\bB^{m-1} \bB x - \bB^{m-1} _n \bB x \|_{l^1 (\mN_0)} = \lim_{n \to \infty}\|\bB^{m-1} y - \bB^{m-1} _n y \|_{l^1 (\mN_0)}= 0,\quad y = \bB x.
		\end{equation*}
		
		The application of Lemma \ref{semigroup} yields \begin{equation*}
			\|\bB^{m-1} _n \bB' _n x \|_{l^1 (\mN_0)} \leq \|\bB^{m-1}_n \|_{l^1(\mN_0)} \| \bB' _n x\|_{l^1 (\mN_0)} \leq \| \bB' _n x\|_{l^1 (\mN_0)}.
		\end{equation*}
		
		Using equation (\ref{definitionsemibist}), we obtain  \begin{equation*}
			\begin{split}
				\| \bB' _n x\|_{l^1 (\mN_0)} &= \sum_{j=0}^{n-1} \bigg|\sum_{s=n}^{\infty} \bB_{js} x_s \bigg| + \sum_{j=n}^{\infty} \bigg| \sum_{k=0}^{\infty} \bB_{jk} x_k \bigg|  \\
				&\leq \sum_{j=0}^{n-1} \sum_{k=n}^{\infty} \bB_{jk} |x_k| + \sum_{j=n}^{\infty} \sum_{k=0}^{\infty} \bB_{jk} |x_k| \leq  \sum_{k=n}^{\infty} |x_k| + \sum_{j=n}^{\infty} \big(\bB |x|\big)_{j}, \text{ where}
			\end{split}
		\end{equation*}
		$|x| = (|x_0|,|x_1|,\ldots) \in \l^{1}(\mN_0)$. Then $\lim_{n \to \infty}\| \bB' _n x\|_{l^1 (\mN_0)} = 0$.
		
		Thus we proved the second assertion.

	\end{proof}

	\subsection{Monotonicity of sequence $\{\mfh(U_J)\}_{J\in\mN_0}$}
	Consider the sequence $\{\mfh(U_J)\}_{J \in \mN_0}$, where $\mfh(U_J)$ is calculated by Theorem \ref{entropyoffinitedim}.

	\begin{theorem}
		Suppose $U \in \calCU_{\mu}(\calH)$. Then $\mfh(U_J) \leq \mfh(U_{J+1}).$
	\end{theorem}
	\begin{proof}
		First, let us prove for all $\alpha \in \{0,\ldots,J-1\}$ \begin{equation} \label{inequalityUalpha}
			 u_{\alpha}(U_J) \leq u_{\alpha}(U_{J+1}). 
		\end{equation}
		
			Using equation (\ref{sequenceA}), we get  \begin{equation*}
			\begin{split}
				a_{\alpha}^k(U_J) &= \sum_{j_{k-1},\ldots,j_0 \in \mZ_J} (\bB_J)_{j_{k-1}j_{k-2}} \ldots (\bB_{J})_{j_0 \alpha} \\ 
				& = \sum_{j_{k-1},\ldots,j_0 \in \mZ_J} (\bB_{J+1})_{j_{k-1} j_{k-2}} \ldots (\bB_{J+1})_{j_0\alpha} \\
				&\leq  \sum_{j_{k-1},\ldots,j_0 \in \mZ_{J+1}} (\bB_{J+1})_{j_{k-1} j_{k-2}} \ldots (\bB_{J+1})_{j_0\alpha} = a_{\alpha}^{k}(U_{J+1}).
			\end{split}
		\end{equation*}
		
		If we combine this with (\ref{sequenceUalpha}), we get  $	u_{\alpha}^n(U_J) \leq u_{\alpha}^n(U_{J+1}).$ By (\ref{numbersUalpha}), we obtain (\ref{inequalityUalpha}). Secondly, we prove
		 \begin{equation}\label{vineq}
			\alpha \in \{0,\ldots,J-1\},\qquad v_{\alpha}(U_J) \leq v_{\alpha}(U_{J+1}).
		\end{equation}	
	
	Using Lemma \ref{aboutPker}, we get \begin{equation*}
			\frac{\bB_{J}^0 + \ldots + \bB_{J}^{n-1}}{n} \to \calP_{\Ker(\bB_{J}-I)},\quad (n\to \infty),\qquad b(U_J)=\bB_{J}.
		\end{equation*}
		
		Let $	\kappa_{J} = (\mu_0,\ldots,\mu_{J-1})$, $\kappa_{J+1} = (\mu_0,\ldots,\mu_{J-1},\mu_J).$
		For all $\alpha \in \{0,\ldots,J-1\}$  \begin{equation*}
			\begin{split}
				(\bB_J ^k \kappa_J)_s &= \sum_{j\in\mZ_J} (\bB_J ^k)_{s j} \mu_j=\sum_{j \in \mZ_J} \sum_{j_k,\ldots,j_1 \in \mZ_J} (\bB_J )_{s j_k}\ldots (\bB_J )_{j_1 j} \mu_j\\ 
				&=\sum_{j \in \mZ_J} \sum_{j_k,\ldots,j_1 \in \mZ_J} (\bB_{J+1} )_{s j_k}\ldots (\bB_{J+1} )_{j_1 j} \mu_j \\
				&\leq \sum_{j \in \mZ_{J+1}} \sum_{j_k,\ldots,j_1 \in \mZ_{J+1}} (\bB_{J+1} )_{s j_k}\ldots (\bB_{J+1} )_{j_1 j} \mu_j = (\bB_{J+1}^k \kappa_{J+1})_s.
			\end{split}
		\end{equation*}
		
		Thus, we obtain \begin{equation*}
			s \in \mZ_J,\qquad \bigg(\frac{1}{n}\sum_{k=0}^{n-1}\bB_{J+1}^k \kappa_{J+1} - \frac{1}{n} \sum_{k=0}^{n-1} \bB_{J}^k \kappa_{J} \bigg)_s \geq 0.
		\end{equation*}
		
		Hence,	$\big(\calP_{\Ker(\bB_{J+1} -I)} \kappa_{J+1} - \calP_{\Ker(\bB_{J} -I)} \kappa_{J} \big)_s \geq 0$ as $n \to \infty$.
		Finally, using  (\ref{inequalityUalpha}) and (\ref{vineq}), we get \begin{equation*}
			\begin{split}
				\mfh(U_J) &= - \sum_{j,k \in \mZ_J} u_{j}(U_J) (\bB_J)_{jk} v_{k}(U_J) \log (\bB_J)_{jk} \\
				&\leq  - \sum_{j,k \in \mZ_J} u_{j}(U_{J+1}) (\bB_{J+1})_{jk} v_{k}(U_{J+1}) \log (\bB_{J+1})_{jk}\\
				&\leq - \sum_{j,k \in \mZ_{J+1}} u_{j}(U_{J+1}) (\bB_{J+1})_{jk} v_{k}(U_{J+1}) \log (\bB_{J+1})_{jk} =\mfh(U_{J+1}).
			\end{split}
		\end{equation*}
	\end{proof}

	\section{Function $\mfH$ and examples} \label{entrexamples}
	
	\begin{definition}
		We put
		\begin{equation}\label{defHu}
			\mathfrak{H}(U) =  \lim_{n \to \infty} \mfh(U_n),\qquad U_n = p_nUq_n.
		\end{equation}
	\end{definition}

	\subsection{Diagonal operator}
	\begin{lemma}
		Let $\calD \in \calCU_{\mu}(\calH)$ be a diagonal operator. Then $\mfH(\calD) = \mfh(\calD) = 0.$
	\end{lemma}
	\begin{proof}
		Consider the partition $\chi_J = \{\{0\},\{1\},\ldots,\{J-1\},X_J\}$, where for all $m \in \mZ_J$ $X_m = \{m\}$, $X_J = \{J,J+1,\ldots\}$. Let $\calD = \text{diag}(d_0,d_1,\ldots)$, where $|d_j| \leq 1$, $\sigma \in \calS_{n,J}$.
		
		Using (\ref{calcI}), we get \begin{equation*}
			\calI_{\calD} (\bG_{\sigma}) = \sum_{j \in \mN_0} |d_j|^2 \mu_j \one_{X_{\sigma(n)}} (j) \ldots\one_{X_{\sigma(0)}} (j).
		\end{equation*}

		Suppose $\calI_{\calD} (\bG_{\sigma}) \neq 0$. Then $	\sigma(n) = \sigma(n-1) = \ldots = \sigma(0).$ Therefore, we have \begin{equation} \label{calculatehuchi}
			\mfh(U,\chi_J,n) = - \sum_{j=0}^{J-1} |d_j|^2 \mu_j \log |d_j|^2 \mu_j - \Bigg(\sum_{j =J}^{\infty} |d_j|^2 \mu_j \Bigg) \log \Bigg(\sum_{j =J}^{\infty} |d_j|^2 \mu_j \Bigg).
		\end{equation}
		
		Since (\ref{calculatehuchi}), it follows that $\mfh(U,\chi_J) = 0.$
		
	\end{proof}

	\subsection{Koopman operator}
	
	\begin{lemma}
		Let $F\colon \mN_0 \to \mN_0$ be an automorphism. Then $\mfH(U_F) = \mfh(U_F)=0.$
	\end{lemma}
	\begin{proof}
	Let us consider the following partition $\chi_{\odot} = \{\{0\},\{1\},\ldots\}$. Suppose $\sigma\colon \mN_0 \to \{0,\ldots,n-1\}$. By Lemma \ref{calcI}\begin{equation}\label{IUF}
		\calI_{U_F} (\bG_{\sigma}(\chi_{\odot})) = \delta_{F(\sigma(n)) \sigma(n-1)} \delta_{F(\sigma(n-1)) \sigma(n-2)}\ldots \delta_{F(\sigma(1)) \sigma(0)} \mu_{\sigma(0)}.
	\end{equation}

Using equation (\ref{IUF}), we obtain \begin{equation*}
	\sigma(k) = F^{-k}(\sigma(0)),\qquad k\in \{1,\ldots,n\}.
\end{equation*}

By definition of entropy (see \ref{limitofentr}), we obtain \begin{equation*}
	\mfh(U_F) =  \lim_{n\to \infty} S_n ,\qquad S_n = - \frac{1}{n}\sum_{j=0}^{n-1} \mu_j \log \mu_j.
\end{equation*}

The application of Jensen's inequality for the function $x \mapsto x\log x$ yields \begin{equation*}
		S_n \leq -\frac{1}{n} M_n \log \frac{1}{n} M_n,\qquad M_n = \sum_{j=0}^{n-1} \mu_j.
\end{equation*}

Obviously for any $n \in \mN_0\colon$ $M_n < M_{n+1},\qquad M_n < 1$. Therefore \begin{equation*}
	\lim_{n\to \infty}\bigg( -\frac{1}{n} M_n \log \frac{1}{n} M_n\bigg) = 0,\qquad \lim_{n \to \infty} S_n =0.
\end{equation*}

By Theorem \ref{entropyoffinitedim}, we have for any $n\in \mN_0\colon$ $\mfh\big((U_F)_n \big)=0$. 
	\end{proof}

	\subsection{Further examples}
	Consider the sequence $\{b_n\}_{n \in \mN_0}$ such that \begin{equation*}
		\sum_{n\in \mN_0} b_n =1,\qquad  b_n \geq 0,\qquad S(b)=-\sum_{n\in \mN_0} b_n \log b_n < \infty.
	\end{equation*}
	
	Suppose $s \in \mN_0$ such that $b_s = \max_{l\in \mN_0} b_l$. We determine $N \in \mN_0$ by the conditions $	N b_s \leq 1,\; (N+1) b_s >1.$ Consider the following operator on $l^{1}(\mN_0)$ with $N$ nonzero columns \begin{equation} \label{operatorB1}
		A=\begin{pmatrix}
			b_0 & b_0 & \cdots & b_0 & 0 & \cdots  \\
			b_1 & b_1 & \cdots & b_1 & 0 & \cdots \\
			\vdots & \vdots & \vdots & \vdots & \vdots & \cdots  \\
			b_n & b_n & \cdots & b_n & 0 & \cdots \\
			\vdots  & \vdots  & \vdots & \vdots & \vdots  & \ddots
		\end{pmatrix}.
	\end{equation}
	
	\begin{lemma}
		Consider the operator (\ref{operatorB1}). Suppose $S(\mu) = - \sum_{j \in \mN_0} \mu_j \log \mu_j < \infty$. Then $\mfH(A) = \mfh(A) = 0.$
	\end{lemma}
	\begin{proof}
		Using Lemma \ref{monot}, we calculate entropy $\mfh(B_1)$ on the finest partition of $\mN_0$ \begin{equation*}
			\chi_{\odot} = \{ \{0\}, \{1\}, \ldots\}.
		\end{equation*}
		
		Using (\ref{operatorB1}), we obtain \begin{equation*}
			\calI_{A} (\bG_{\sigma} (\chi_{\odot})) =
			b_{\sigma(n)} \ldots b_{\sigma(1)} \mu_{\sigma(0)} \one _{\mZ_N} (\sigma(n-1)) \ldots \one _{\mZ_N} (\sigma(0)).
		\end{equation*}
		
		We put $\mZ_N = \{0,\ldots,N-1\}$. We have \begin{equation*}
			\begin{split}
				\mfh(A,\chi_{\odot},n) &= -\sum_{j_n,\ldots,j_0 \in \mN_0} \prod_{s=1}^n b_{j_s} \mu_{j_0} \prod_{s=0}^{n-1} \one _{\mZ_N} (j_s) \log \bigg( \prod_{s=1}^n b_{j_s} \mu_{j_0} \prod_{s=0}^{n-1} \one _{\mZ_N} (j_s)\bigg) \\
				&=A_1 + A_2 + A_3,
			\end{split}
		\end{equation*}
		where \begin{equation*}
			\begin{split}
				&A_1 = - \sum_{j_n,\ldots,j_0 \in \mN_0} \prod_{s=1}^n b_{j_s} \mu_{j_0}\prod_{s=0}^{n-1} \one _{\mZ_N} (j_s) \log b_{j_n}, \\
				&A_2 = - \sum_{k=1}^{n-1} \sum_{j_n,\ldots,j_0 \in \mN_0} \prod_{s=1}^{n} b_{j_s} \mu_{j_0} \prod_{s=0}^{n-1} \one _{\mZ_N} (j_s) \log \big( b_{j_k} \one _{\mZ_N} (j_k)\big),\\
				&A_3 =- \sum_{j_n,\ldots,j_0 \in \mN_0} \prod_{s=1}^n b_{j_s} \mu_{j_0} \prod_{s=0}^{n-1} \one _{\mZ_N} (j_s) \log \big(\mu_{j_0} \one _{\mZ_N} (j_0)\big).
			\end{split}
		\end{equation*}
		
		By definition, we put \begin{equation*}
			S_N = \sum_{j \in \mZ_N} b_j, \quad M_N = \sum_{j \in \mZ_N} \mu_j.
		\end{equation*}
		
		Consider $A_1, A_2, A_3$ \begin{equation*}
			A_1 = -S_N ^{n-1} M_N \sum_{j \in \mN_0} b_j \log b_j = -S_N ^{n-1} M_N S(b),
		\end{equation*}
		
		\begin{equation*}
			A_2 = - S_N ^{n-2} M_N (n-1) \sum_{j\in \mZ_N} b_j \log b_j, \qquad A_3 = - S_N ^{n-1} \sum_{j\in \mN_0} \mu_j \log \mu_j = - S_N ^{n-1} S(\mu).
		\end{equation*}
		
		Therefore, we obtain \begin{equation*}
			\lim_{n \to \infty} \frac{A_1}{n} = 0,\quad \lim_{n \to \infty} \frac{A_3}{n} = 0,\quad \lim_{n \to \infty} \frac{A_1}{n} = 0,\qquad \mfh(A) = 0.
		\end{equation*}
		
		Consider the sequence $\{a_{\alpha} ^k (A_n)\}$ is defined by (\ref{sequenceA}). If $n>N$. Then \begin{equation*}\begin{split}
				a_{\alpha} ^k (A_n) &= \sum_{j_{k-1},\ldots,j_0 \in \mZ_n} A_{j_{k-1} j_{k-2}} \ldots A_{j_0 \alpha} =\sum_{j_{k-1},\ldots,j_0 \in \mZ_n} \prod_{s=1}^{k-1} b_{j_s} \prod_{s=1}^{k-2} \one_{\mZ_N} (j_s) \one_{\mZ_N} (\alpha)\\
				&= S_n S_N ^{k-2} \one_{\mZ_N} (\alpha).
			\end{split}
		\end{equation*}

		Using Corollary \ref{lemmaaboutualpha}, we obtain \begin{equation*}
			u_{\alpha} (A_n) = \lim_{k \to \infty} S_n S_N^{k-1} \one_{\mZ_N} (\alpha) = 0.
		\end{equation*}
		
		Consequently, we have $	\mfh(A_n)=0$, $ \mfH(A) = 0.$
		
	\end{proof}
	
	Consider the sequence $\{\alpha_n\}_{n \in \mN_0}$ such that $	\sum_{n \in \mN_0} \alpha_n \leq 1,\;\alpha_n \geq 0.$ Let $B$ the operator of $l^1(\mN_0)$ to $l^1(\mN_0)$ such that \begin{equation}\label{theoperatorBalpha}
		B=\begin{pmatrix}
			\alpha_0 & \alpha_1 & \alpha_2 & \cdots   \\
			\alpha_1 & \alpha_2 & \alpha_3 & \cdots  \\
			\alpha_2 & \alpha_3 & \alpha_4 & \cdots\\
			\vdots  & \vdots  & \vdots  & \ddots
		\end{pmatrix}.
	\end{equation}
	
	\begin{lemma}
		Consider the operator (\ref{theoperatorBalpha}). Then $\mfH(B)=0.$
	\end{lemma}
	\begin{proof}
		It is obvious that the operator $B_n$, $n\in \mN_0$ is strict contraction. Then $\mfh(B_n) = 0,$ $\mfH(B)=0.$
	\end{proof}

	\subsection{Non-approximable operator}
	
	Consider the following operator on $l^2(\mN_0,\mu)$ \begin{equation} \label{operatornonapp}
		D = 
		\begin{pmatrix}
			\sqrt{\frac{1}{2}} & \sqrt{\frac{\mu_1}{\mu_0} \frac{1}{2}} & 0 & \cdots  \\
			0 & \sqrt{\frac{1}{2}} & \sqrt{\frac{\mu_2}{\mu_1} \frac{1}{2}}  & \cdots \\
			0 & 0 & \sqrt{\frac{1}{2}}  & \cdots \\
			\vdots  & \vdots  & \vdots & \ddots
		\end{pmatrix},\qquad b(D) = \begin{pmatrix}
			1/2 & 1/2 & 0  & \cdots  \\
			0 & 1/2 & 1/2 & \cdots \\
			0 & 0 & 1/2 & \cdots \\
			\vdots  & \vdots  & \vdots & \ddots
		\end{pmatrix}.
	\end{equation}

	\begin{lemma}
		Suppose $S(\mu)=-\sum_{j \in \mN_0} \mu_j \log \mu_j < \infty$.	Consider the operator (\ref{operatornonapp}). Then $		\mfH (D) = 0$, $\mfh(D) = \log 2.$
	\end{lemma}
	\begin{proof}
		Consider the function (\ref{defmfhUchin}) on the finest partition $\chi_{\odot} = \{ \{0\},\{1\},\ldots\}$. Let $\delta$ be Kronecker delta. 	Let $\calS_{n,\infty}$ be the set of maps $\mZ_{n+1} \to \mN_0$. Now, by (\ref{Ionfinest})
		\begin{equation} \label{calcIA}
			\begin{split}
				\calI_D(\bG_{\sigma}(\chi_{\odot})) &= ( 1/2 )^n \mu_{\sigma(0)}	 (\delta_{\sigma(1), \sigma(0)} + \delta_{\sigma(1), \sigma(0)-1})\times\\
				&\qquad\times (\delta_{\sigma(2), \sigma(1)} + \delta_{\sigma(2),\sigma(1)-1})  \ldots(\delta_{\sigma(n),\sigma(n-1)} +  \delta_{\sigma(n),\sigma(n-1)-1})
			\end{split}
		\end{equation}
		
		Using (\ref{calcIA}), we get \begin{equation*}
			\mfh(D,\chi_{\odot},n) = - 2^n \sum_{j \in \mN_0} \frac{1}{2^n} \mu_j \log \frac{1}{2^n} \mu_j = n \log 2 - \sum_{j \in \mN_0}\mu_j \log \mu_j.
		\end{equation*}
		
		Therefore, we have $\mfh(D,\chi_{\odot}) = \log 2.$ Using Lemma \ref{monot}, we get $\mfh(D) = \log 2$.	Consider the operator $D_n = q_n A p_n$, where $p_n$ defined by (\ref{operatorsPandQ}). Consider a $x \in \mC ^n$ such that $A_n x = x$. Thus we have \begin{equation} \label{systemforker}
			\begin{cases}
				1/2 \,(x_0 + x_1) = x_0, \\
				1/2 \,(x_1 + x_2) = x_1, \\
				\qquad\qquad	\vdots \\
				1/2\,(x_{n-2} + x_{n-1}) = x_{n-2}, \\
				1/2\, x_{n-1} = x_{n-1}.
			\end{cases}
		\end{equation}
		
		Hence by (\ref{systemforker}), we have \begin{equation}\label{equationaboutker0}
			x=0,\qquad \text{ Ker} (b(D_n) - I) = \{0\}.
		\end{equation}
		
		Using Theorem \ref{entropyoffinitedim} and (\ref{equationaboutker0}), we get $	\mfh(D_n) = 0, \; \mfH(D) = 0.$
	\end{proof}

	\subsection{Operator $\bB_{\alpha}$}
	
	By definition, put \begin{equation*}
		\bB_n=\begin{pmatrix}
			\frac{1}{n}& \cdots & \frac{1}{n}   \\
			\vdots & \ddots & \vdots \\ 
			\frac{1}{n} &\cdots &\frac{1}{n} 
		\end{pmatrix},\qquad n \in \mN_0.
	\end{equation*}
	
	Consider the sequence $\{\alpha_n\}_{n \in \mN_0}$ such that for any $j \in \mN_0$: $\alpha_{j+1} \geq \alpha_{j}$, $\alpha_j \in \mN_0.$ Let $\bB_{\alpha}$ be the operator on $l^{1}(\mN_0)$ defined by \begin{equation} \label{theoperatorBdiagalpha}
		\bB_{\alpha} = \text{diag}(\bB_{\alpha_1},\bB_{\alpha_2},\ldots).
	\end{equation}
	
	Let $A_s$ be the number defined by $A_s = \sum_{k=1}^s \alpha_k.$

	\begin{lemma} \label{lemmaexampleentropy}
		Consider the operator (\ref{theoperatorBdiagalpha}). Then \begin{equation*}
			\mfH(\bB_{\alpha}) = \sum_{k=1}^{\infty} \log \alpha_k \sum_{s= A_{k-1}}^{A_k -1} \mu_s.
		\end{equation*}
	\end{lemma}
	\begin{proof}
		
		Consider the operator $(\bB_{\alpha})_{A_k} = \text{diag} (\bB_{\alpha_{1}},\ldots,\bB_{\alpha_k}).$ Let us calculate $u_j ((\bB_{\alpha})_{A_k})$ and $v_s ((\bB_{\alpha})_{A_k})$. Note that for any $n\in \mN_0$: $\big((\bB_{\alpha})_{A_k} \big)^n = (\bB_{\alpha})_{A_k}.$ Therefore, we have
		\begin{equation*}
			\calP_{\text{Ker}((\bB_{\alpha})_{A_k} - I)} = \lim_{n\to \infty} \frac{1}{n}\sum_{k=0}^{n-1} \big((\bB_{\alpha})_{A_k} \big)^k =  (\bB_{\alpha})_{A_k}.
		\end{equation*}
		
		Suppose $A_m \leq s \leq A_{m+1} -1$. Then \begin{equation} \label{calculatevs}
			v_s ((\bB_{\alpha})_{A_k}) = \big(\calP_{\text{Ker}((\bB_{\alpha})_{A_k} - I)} E_{\mu}\big)_s = \frac{1}{\alpha_m} \sum_{n=A_m}^{A_{m+1} -1} \mu_n.
		\end{equation}
		
		The matrix $(\bB_{\alpha})_{A_k}$ is the double stochastic matrix. Hence \begin{equation} \label{calculateus}
			u_j((\bB_{\alpha})_{A_k}) = 1.
		\end{equation}
		
		Using (\ref{calculatevs}), (\ref{calculateus}), we obtain \begin{equation*}
			\mfh((\bB_{\alpha})_{A_n}) = \sum_{k=1}^n \log \alpha_k \sum_{s= A_{k-1}} ^{A_k -1} \mu_s. 
		\end{equation*}
		
		Therefore, we have  \begin{equation*}
			\mfH(B_{\alpha}) =  \lim_{n \to \infty} \sum_{k=1}^n \log \alpha_k \sum_{s= A_{k-1}} ^{A_k -1} \mu_s = \sum_{k=1}^{\infty} \log \alpha_k \sum_{s= A_{k-1}} ^{A_k -1} \mu_s. 
		\end{equation*}
	\end{proof}
	
	\begin{lemma}
		Suppose that for any $\alpha_0=\alpha_1=\ldots = n \in \mN_0$. Then \begin{equation*}
			\mfH(\bB_{\alpha}) = \log n.
		\end{equation*}
	\end{lemma}
	
	\begin{proof}
		Using Theorem \ref{lemmaexampleentropy}, we obtain \begin{equation*}
			A_k = k n,\qquad \mfH(\bB_{\alpha})= \sum_{k=1}^{\infty} \log n \sum_{s = (k-1)n}^{kn -1} \mu_s = \log n.
		\end{equation*} \end{proof}
	
	Let us consider the following conditions  \begin{equation} \label{condit2}
		\alpha_s = 2^s,\quad s\in \mN_0,
	\end{equation}
	\begin{equation}\label{condit1}
		\mu_s = \begin{cases}
			\frac{C}{2 (k+1)^{3/2}} & \text{if } s = A_k, \\
			\frac{1}{A_{k+1} - A_{k} -1}\frac{C}{2(k+1)^2} & \text{if $s\in(A_k;A_{k+1} -1 )$},
		\end{cases} \qquad C = \sum_{k=0}^{\infty}	\frac{1}{(k+1)^{3/2}}.
	\end{equation}
	
	\begin{lemma}
		Under the conditions of (\ref{condit2}) and (\ref{condit1}), we have $			\mfH(\bB_{\alpha}) = \infty.$
	\end{lemma}
	
	\begin{proof}
		Using Lemma \ref{lemmaexampleentropy} and conditions (\ref{condit2}) and (\ref{condit1}), we obtain \begin{equation*}
			\mfH(\bB_{\alpha}) = \sum_{k=1}^{\infty} \log 2^k \sum_{s=A_{k-1}}^{A_k -1} \mu_s  = \log 2 \sum_{k=1}^{\infty} k \bigg(\frac{C}{2k^{3/2}} + \frac{C}{2k^{3/2}} \bigg)= C \log 2 \sum_{k=1}^{\infty} \frac{1}{k^{1/2}} = \infty.
		\end{equation*}
	\end{proof}

\end{document}